\documentclass{siamltex}
\usepackage{amsmath,amsfonts,amssymb,mathtools}
\usepackage{stmaryrd}
\usepackage{latexsym}
\usepackage{epsfig,overpic}
\usepackage{color}
\usepackage{cite}
\usepackage{pgfplots}
\usepackage{pgfplotstable}
\pgfplotsset{compat=1.11}
\usepackage{booktabs}
\usepackage{algorithm}
\usepackage{subcaption}             % for subfigure enviroment

\let\ep\varepsilon

\newcommand{\bB}{\mathbf B}

\newcommand{\bI}{\mathbf I}
\newcommand{\bP}{\mathbf P}

\newcommand{\bU}{\mathbf U}
\def\enorm#1{|\!|\!| #1 |\!|\!|}
\newcommand{\bx}{\mathbf x}
\newcommand{\bg}{\mathbf g}
\newcommand{\bn}{\mathbf n}
\newcommand{\tbn}{\widetilde{\mathbf n}}
\newcommand{\bp}{\mathbf p}

\newcommand{\bw}{\mathbf w}
\newcommand{\be}{\mathbf e}
\newcommand{\by}{\mathbf y}
\newcommand{\bz}{\mathbf z}

\newcommand{\bb}{\mathbf b}
\newcommand{\bu}{\mathbf u}
\newcommand{\bv}{\mathbf v}
\newcommand{\blf}{\mathbf f}
\newcommand{\bH}{\mathbf H}
\newcommand{\bF}{\mathbf F}

\newcommand{\wn}{{w_N}}

\newcommand{\G}[2]{\Gamma^{#1}_{#2}}
\newcommand{\norm}[1]{\Vert #1 \Vert}

\newcommand{\consist}{\mathcal{E}_C^n}
\newcommand{\interpol}{\mathcal{E}_I^n}
\newcommand{\err}{\mathbb{E}}

\newcommand{\T}{\mathcal T}
\newcommand{\Div}{\operatorname{\rm div}}
\newcommand{\DivG}{{\,\operatorname{div_\Gamma}}}
\newcommand{\nablaG}{\nabla_\Gamma}

\newcommand{\cO}{\mathcal{O}}

\newcommand{\rr}{\mathbb{R}}

\newcommand{\Gs}{\mathcal{G}} %{\Gamma_\ast}

\newcommand{\pc}{\partial^{\circ}}
\newcommand{\divG}{{\mathop{\,\rm div}}_{\Gamma}}

\newcommand{\gradG}{\nabla_{\Gamma}}

\newcommand{\Gh}{\Gamma_h}
\newcommand{\Ghn}{\Gamma_h^n}
\newcommand*\diff{\mathop{}\!\mathrm{d}}
\newcommand{\vect}[1]{\boldsymbol{\mathbf{#1}}}

\newcommand{\oGn}{\omega_{\Gamma}^n}
\newcommand{\normalbar}{\overline{\bn}_h}
\usepackage{booktabs}
\usepackage{accents}
\renewcommand*{\dot}[1]{%
	\accentset{\mbox{\large\bfseries .}}{#1}}

\newcounter{ass}

\newcommand{\ifasslabel}[1]{}

\newtheorem{assumption}{Assumption}[section]

\newtheorem{remark}{Remark}[section]

\def\enormu#1{|\!|\!| #1 |\!|\!|_{U^n}}

\pgfplotstableset{
    ColumnStyleEOC/.style={/pgfplots/table/display columns/#1/.style={%
        column name = EOC, column type/.add={|}{||},
    }},
    ColumnStyleError/.style={/pgfplots/table/display columns/#1/.style={%
        sci zerofill, sci,
    }},
}

\begin{document}
\title{An Eulerian finite element method for tangential Navier-Stokes equations on evolving surfaces}

\author{ Maxim A. Olshanskii\thanks{Department of Mathematics, University of Houston, Houston, Texas 77204-3008, USA
maolshanskiy@uh.edu, www.math.uh.edu/\string~molshan}
\and Arnold Reusken\thanks{Institut f\"ur Geometrie und Praktische Mathematik, RWTH-Aachen University, D-52056 Aachen,
	Germany (reusken@igpm.rwth-aachen.de)} \and Paul Schwering\thanks{Institut f\"ur Geometrie und Praktische Mathematik, RWTH-Aachen University, D-52056 Aachen,
	Germany (schwering@igpm.rwth-aachen.de)}}

\maketitle

\begin{abstract}
The paper introduces a geometrically  unfitted finite element method for the numerical solution of the tangential Navier--Stokes equations posed on a passively evolving smooth closed surface embedded in $\mathbb{R}^3$. The discrete formulation employs  finite difference and finite elements methods to handle evolution in time and variation in space, respectively. A complete  numerical analysis of the method is presented, including stability, optimal order convergence, and quantification of the geometric errors. Results of numerical experiments are also provided.
\end{abstract}

% REQUIRED

\section{Introduction}
Lipid membranes and liquid crystal shells are examples of  deformable thin structures exhibiting lateral fluidity~\cite{dimova2006practical,cicuta2007diffusion,keber2014topology}.
Continuum based modeling of such materials leads to systems of PDEs posed on evolving surfaces embedded in $\mathbb{R}^3$.
For example, the motion of an inextensible viscous fluid layer represented by a material surface is governed by  the evolving surface Navier--Stokes equations  derived in e.g., \cite{Yavarietal_JNS_2016,Gigaetal,Jankuhn1,miura2017singular,reuther2018solving}.
While the evolving surface Navier--Stokes system was discussed in the literature, to the best of our knowledge there is no existing  well-posedness or numerical analysis of this problem.
The recent paper~\cite{olshanskii2022tangential} addressed well-posedness  of a simplified problem consisting of \emph{tangential}  surface Navier--Stokes equations (TSNSE)  posed on a \emph{passively} evolving surface embedded in $\mathbb{R}^3$.
A weak variational formulation of TSNSE was shown to be well-posed for any finite final time and without smallness conditions on the data. For that  variational formulation we introduce  a discretization method and prove its stability and optimal order convergence. This is the first study addressing numerical analysis of a fluid PDE system posed on an evolving surface.

For  discretization of the TSNSE we consider a geometrically unfitted finite element method known as TraceFEM~\cite{ORG09,olshanskii2016trace}. The TraceFEM applies to a fully Eulerian  formulation of the problem and does not require a surface triangulation, which makes it convenient for  deforming surfaces. In TraceFEM one uses standard  (bulk) finite element spaces  to approximate unknown quantities on the surface $\Gamma(t)$ which propagates through  a given triangulation of an ambient volume $\Omega$, i.e.  $\Gamma(t)\subset\Omega$ for all times $t$.  The discrete formulation does not need a surface parametrization and uses tangential calculus in the embedding space $\mathbb{R}^3$. For scalar PDEs on evolving surfaces, a space--time and a hybrid  (finite difference in time -- finite elements in space) variant of TraceFEM are known in the literature~\cite{olshanskii2014eulerian,lehrenfeld2018stabilized}. For the TSNSE we choose the hybrid approach since it is more flexible in terms of implementation and the choice of elements.

A variant of the hybrid FEM  that we consider in this paper has recently been applied to time-dependent Stokes equations in a moving volume domain $\Omega(t) \subset \Bbb{R}^3$  in \cite{Lehrenfeldetal2021,burman2022eulerian}. In \cite{Lehrenfeldetal2021} Taylor-Hood finite elements are used, whereas in  \cite{burman2022eulerian} equal order finite element spaces combined with a continuous interior penalty pressure stabilization are applied for space discretization. Both papers present a complete discretization error analysis. The resulting error bounds in \cite{Lehrenfeldetal2021} are suboptimal both for velocity and pressure and in   \cite{burman2022eulerian} the bound for the pressure error is suboptimal. In both papers it is mentioned that the suboptimality of the bounds is essentially due to the lack of a uniform discrete pressure stability bound, cf \cite[Remark 5.11]{Lehrenfeldetal2021}, \cite[Section 4.1]{burman2022eulerian}. Related to this we mention a further
new contribution  of this paper.
%Although the case of volumetric domains is not discussed here,
We introduce, in the surface case, a new argument  which leads to a discrete pressure stability bound that is uniform in the parameter range $h^2 \lesssim \Delta t$, cf. Remark~\ref{remArgument}.  Using this we derive error estimates for velocity and pressure that are optimal in the parameter range  $h^2 \lesssim \Delta t \lesssim h$. These parameter range restrictions are reasonable if one considers BDF1 or BDF2 time discretization methods and low order finite element pairs, e.g. Taylor-Hood $P_2$-$P_1$.

Concerning related work on  numerical analysis and development of computational methods for fluid equations posed on surfaces we note the following.  In the past few years there was an increasing  interest is this research field; see, e.g.,
\cite{reuther2015interplay,reusken2018stream, reuther2018solving,fries2018higher,nitschke2019hydrodynamic,gross2020meshfree,bonito2020divergence,lederer2020divergence,suchde2021meshfree,de2021numerical,rank2021active,sun2022modeling,brandner2022finite,olshanskii2018finite,olshanskii2019penalty,olshanskii2021inf,jankuhn2021error,olshanskii2022tangential}. In many of these papers the main topic is the development of numerical methods. The papers that focus on numerical (error) analysis mostly
 address the ``simple'' case of a homogeneous viscous surface fluid flow on a \emph{steady} smooth surface. For the scope of this paper, the most relevant results can be found in \cite{lederer2020divergence,bonito2020divergence,olshanskii2021inf,jankuhn2021error}, where  first stability and error analyses of finite element discretizations for the Stokes problem on a steady surface are presented.
The authors of \cite{lederer2020divergence,bonito2020divergence} analyzed  $H$(div)-conforming finite elements for the surface Stokes equation.   The surface Stokes problem  has been discretized using unfitted $P_2$-$P_1$ elements in  \cite{olshanskii2021inf}, and with higher order Taylor-Hood bulk elements \cite{jankuhn2021error}.
%The last two papers give full convergence analysis, but assume exact numerical integration over the curvilinear surface. This assumption was relaxed in \cite{jankuhn2021error}, where extension to higher order trace Taylor--Hood-type elements was also done.
We also mention the work in \cite{Jankuhn2,Jankuhn2020,hansbo2020analysis}, where the surface FEM of Dziuk~\cite{Dziuk88} and TraceFEM are analyzed for the surface vector-Laplace problem, which is closely related to the surface Stokes problem. Here we will use results from these papers for  the analysis of the geometric error.

The rest of the paper is organized in four sections. In section~\ref{s:form} we introduce necessary notations from tangential calculus and formulate the  TSNSE as a well-posed variational problem. Section~\ref{s:discretization} presents the discretization method for this  problem. We first explain the idea of how the system is integrated numerically in time and proceed to a fully discrete method.
Section~\ref{s:FEM} contains the main results. The error analysis presented in this section is rather long and technical but its structure is canonical. We derive continuity and stability estimates for the discrete problem. Substituting the continuous solution in the discrete variational formulation results in consistency terms for which bounds are derived.
A priori discretization error bounds are derived by using an established approach, based on discrete stability, consistency error bounds and interpolation error bounds. In Remark~\ref{RemAnalysis} we give more explanation concerning the structure and the key new ingredients of the error analysis.
 The main result of the paper is Theorem~\ref{Th2AR} which yields optimal order error estimates for the velocity in an energy norm and for pressure in a special $H^1$-type norm.
Results of a numerical experiment presented in section~\ref{s:Numerics}  illustrate the optimal order convergence of the method.

\section{Problem formulation} \label{s:form}
In this section we explain the tangential surface Navier-Stokes equations that we treat in this paper.
Consider, for $t\in [0,T]$, a \textit{material} surface  $\Gamma(t)$ embedded in $\mathbb{R}^3$ as defined in \cite{GurtinMurdoch75,MurdochCohen79}, with a density distribution $\rho(t,\bx)$.  By $\bu(t,\bx)$, $\bx \in \Gamma(t)$,  we denote the velocity field  of the density flow on $\Gamma(t)$, i.e. $\bu(t,\bx)$ is the velocity of a material point  $\bx\in\Gamma(t)$.
We further assume that the geometric evolution of  $\Gamma(t)$ is determined by a {given} smooth velocity field $\bw=\bw(t,\bx)$, which passively advects the initial surface $\Gamma_0:=\Gamma(0)$:
\begin{equation} \label{defevolving}
\Gamma(t)=\{\by\in\mathbb{R}^3~|~\by=\bx(t,\bz),~\bz \in \Gamma_0 \},
\end{equation}
with the trajectories $\bx(t,\bz)$ being the unique solutions of the Cauchy problem
\begin{equation}\label{def_flow_map}
	\frac{d}{dt} \bx(t,\bz) = \bw(t, \bx(t,\bz)),\quad
		\bx(0,\bz)=\bz
\end{equation}
for all $\bz \in \Gamma_0$.
This induces  the smooth space-time manifold
\[
\Gs=\bigcup\limits_{t\in[0,T]}\{t\} \times \Gamma(t)\subset\mathbb{R}^4.
\]
We need a few notations of geometric quantities and tangential differential operators. For a given $t \in [0,T]$ we write $\Gamma=\Gamma(t)$.
The outward pointing normal vector on $\Gamma$ is denoted by $\bn$. The normal projector on the tangential space at $\bx\in\Gamma$ is given by $\bP=\bP(\bx)=\bI- \bn\bn^T$. For  a scalar function $p:\, \Gamma \to \mathbb{R}$ or a vector field $\bu:\, \Gamma \to \mathbb{R}^3$   their smooth extensions to a neighborhood $\mathcal{O}(\Gamma)$  of $\Gamma$ are denoted by $p^e$ and $\bu^e$, respectively. The surface gradient,  covariant derivative and surface divergence  on $\Gamma$ can be defined through derivatives in $\mathbb{R}^3$   as $\nablaG p=\bP\nabla p^e$, % $D_\Gamma \bu\coloneqq \nabla \bu^e \bP$,
 $\nabla_\Gamma \bu= \bP \nabla \bu^e \bP$, and $\Div_{\Gamma}\bu=\mbox{tr}(\nabla_\Gamma \bu)$. These definitions  are  independent of a particular smooth extension of $p$ and $\bu$ off $\Gamma$.
The surface rate-of-strain tensor \cite{GurtinMurdoch75} is given by
$E_s(\bu)=\frac12(\nabla_\Gamma \bu + \nabla_\Gamma \bu^T)$.
By $\bH=\nabla_\Gamma\bn \in\mathbb{R}^{3\times3}$ we denote the  Weingarten mapping  and by $\kappa:=\mbox{tr}(\bH)$  twice the mean curvature.
For velocity fields on $\Gamma(t)$  we use a splitting  into tangential and normal components
\[
\bu=\bu_T+\bu_N=\bu_T+u_N\bn,\quad \text{with}~ u_N=\bu\cdot \bn.
\]
In our setting, the normal component of $\bu$ is completely determined by the ambient flow $\bw$, i.e.
$\bu_N=\bw_N$  on $\Gamma$. For the normal component of the given smooth velocity field we have $\bw_N=w_N \bn$, wit $w_N=\bw \cdot \bn$ the given smooth normal velocity of the surface $\Gamma$. Besides derivatives on $\Gamma$ we also need the material derivative along $\Gs$, denoted by $\dot{g}$, which is the derivative  of a surface quantity $g$ along the  trajectories $\bx(t,\bz)$ of  material points. It can be written as
$
	\dot g = \frac{\partial g^e}{\partial t}+ (\bu\cdot\nabla) g^e$ {on} $\Gs$. We also introduce the so-called normal time derivative of $g$, denoted by $\pc g$, which describes the variation of $g$ along normal trajectories of points on $\Gamma(t)$:
\begin{equation} \label{normDer}
 \pc g := \dot g - \bu\cdot\nablaG g =  \frac{\partial g^e}{\partial t}+ (\bw_N\cdot\nabla) g^e\quad \text{on}~\Gs.
\end{equation}
For a vector valued quantity we use this definition componentwise.
Conservation of momentum and an inextensibility condition lead to the following system governing the free lateral motion of the  material viscous surface~\cite{Jankuhn1,olshanskii2022tangential}:
For a given density distribution $\rho>0$, viscosity coefficient $\mu>0$,  find the tangential velocity field $\bu_T$ and surface pressure $p$ satisfying the initial condition $\bu_T(0)=\bu_0$ and the system of equations, which we call the {tangential} surface Navier-Stokes equations (TSNSE):
\begin{equation}\label{NSalt}
	\left\{\begin{aligned}
		\rho\big(\,\bP \pc\bu_T + w_N \bH \bu_T +(\nablaG \bu_T) \bu_T\big)- 2\mu  \bP \divG E_s(\bu_T) +\nabla_\Gamma p&=\blf \\
		\divG \bu_T   &= f\\
	\end{aligned}\right. \quad\text{on}~\Gamma(t),
\end{equation}
with right-hand sides known in terms of geometric quantities,  $w_N$ and the tangential component of the external area force $\bb$:
\begin{equation}\label{rhs}
	f = - w_N \kappa,\qquad \blf=\bb_T+ 2\mu  \bP \divG(w_N\bH) + \tfrac\rho2\gradG w_N^2.
\end{equation}
The system can be seen as an idealized model for the motion of a thin fluid layer embedded in  bulk fluid, where one neglects
friction forces between the surface and bulk fluids as well as any effect of the layer on the bulk flow (more precisely, one may assume that a lateral component of the normal bulk stress is given by $\bb_T$).
System \eqref{NSalt}--\eqref{rhs} also appears as an auxiliary problem if one applies directional splitting to the full system of equations governing the evolution of a material inextensible fluidic surface; see~\cite{Jankuhn1}. We further set $\rho=1$.

In this paper, we represent $\Gamma(t)$ as the zero level set of a smooth level-set function $\phi(t,\bx)$,
 \[
\Gamma(t)=\{\bx\in\mathbb{R}^3\,:\,\phi(t,\bx)=0\},
\]
such that $|\nabla \phi|\ge c>0$ in $\mathcal{O}(\Gs)$, a neighborhood of $\Gs$.
To simplify the presentation and analysis, we make the assumption that the \emph{level set function has the signed distance property}. This assumption, however, is not essential. In Remark~\ref{remissues} we comment on the generalization of the method to the case where $\phi$ is not necessarily a signed distance function. %, cf. the introductory part of section \ref{s:stab:semi-disc}.

\section{Discretization method} \label{s:discretization}
In this section we present  a fully Eulerian finite element  method for the TSNSE \eqref{NSalt}. The method is based on the same ideas as used in \cite{lehrenfeld2018stabilized},  namely  a combination of an implicit time stepping scheme with a TraceFEM in space.
We start with the discretization  of the system's evolution in time.

\subsection{Time-stepping scheme}
Consider uniformly distributed time nodes  $t_n=n\Delta t$, $n=0,\dots,N$, with the  time step $\Delta t=T/N$ and $I_n:=[t_{n-1},t_n]$, $1 \leq n \leq N$. We assume that the time step $\Delta t$ is sufficiently small such that
\begin{equation}\label{ass1}
	\Gamma(t_n)\subset\mathcal{O}( \Gamma(t_{n-1})),\quad~n=1,\dots,N, ~ %\text{cf. Fig. \ref{fig:Gammaneighborhood}.}
\end{equation}
with $\mathcal{O}( \Gamma(t))$ a neighborhood of $ \Gamma(t)$ where a smooth  extension  of surface quantities from $\Gamma(t)$  is well defined.  Among smooth extensions of a (scalar or vector valued) function $g$ we now choose the constant extension in normal direction $g^e$, i.e., $\nabla d \cdot \nabla g^e=0$ in $O(\Gamma(t))$, with $d=d(t,\cdot)$ the signed distance function to $\Gamma(t)$. For the normal extension the last term in \eqref{normDer} vanishes and thus we have
\begin{equation} \label{tderivative}
   \pc g = \frac{\partial g^e}{\partial t} \quad \text{on}~\Gs.
\end{equation}
%\textbf{\small MO: Presentation question: Few paragraphs above we say that for the derivation of the method $\phi$ is not necessary a sign distance function, but here we assume it sign distance again. We need either to assume $\phi=d$ from the beginning or add the extra term in the derivation and FE formulation. The choice also depends on the implementation that you use.} \\
Based on this we introduce on $\Gamma(t_n)$ the normal time derivative approximation
\begin{equation} \label{tdiscrete}
 \bP \pc \bu_T = \bP \frac{ \partial \bu_T^e}{\partial t} \approx \frac{\bu_T(t_n)- \bP(t_n) \bu_T(t_{n-1})^e}{\Delta t}.
\end{equation}
Due to \eqref{ass1} $\bu_T(t_{n-1})^e$ is defined on $\Gamma(t_n)$.
For the normal $\bn$ on $\Gamma(t_j)$  its constant  extension   in $\mathcal{O}(\Gamma(t_j))$ is also denoted by $\bn$, i.e.,
$\bn=\nabla d$.
We further use the notation $\bu_T^j$  and $p^j$ for an approximation of $\bu_T(t_j)^e$ and  $p(t_j)$, respectively.  Based on \eqref{tdiscrete} and  \eqref{ass1} we consider the following time discretization method for \eqref{NSalt}. Given $\bu_T^0=\bu_T(0)^e$ in $\mathcal{O}(\Gamma_0)$, for $n=1,\ldots N$, find $\bu^n_T$, defined in $\mathcal{O}(\Gamma(t_n))$ and tangential to $\Gamma(t_n)$, i.e. $(\bu^n_T\cdot\bn)|_{\Gamma(t_n)}=0$,  and $p^n$ defined on $\Gamma(t_n)$ such that
\begin{align}
	\label{NS_FD}
	\left\{\begin{aligned}
		\frac{\bu^n_T-\bP\bu^{n-1}_T}{\Delta t} + w_N^n \bH \bu_T^n + (\nabla_\Gamma\bu_T^n)\bu_T^{n-1}
		- 2\mu  \bP \divG E_s(\bu_T^n)  +\nabla_\Gamma p^n & =\blf^n\\
		\divG \bu_T^n   &= f^n\\
	\end{aligned}\right.\quad&\text{on}~~\Gamma(t_n),\\[1ex]
	\bn\cdot\nabla\bu^n_T  =0~\quad&\text{in}~~\mathcal{O}(\Gamma(t_n)),\label{e:nExt1}
\end{align}
with $w_N^n:=w_N(t_n)$,  $\blf^n:=\blf(t_n)$, $f^n:=f(t_n)$. Note that in the inertia term we also use $\bP\bu_T^{n-1}$, since $(\nabla_\Gamma\bu_T^n)\bu_T^{n-1}=(\nabla_\Gamma\bu_T^n)\bP\bu_T^{n-1}$ holds.
Geometric information in \eqref{NS_FD} is taken for $\Gamma(t_n)$, i.e. $\bn=\bn(t_n)$, $\bP=\bP(t_n)$, $\bH=\bH(t_n)$. In \eqref{NS_FD} we use a BDF1 (i.e., Euler implicit) type time discretization, which has local truncation error $\mathcal{O}(\Delta t^2)$. Related to that we use a simple, first order in $\Delta t$ accurate, linearization of the quadratic nonlinearity, i.e., $(\nabla_\Gamma\bu_T^n)\bu_T^{n}$ is replaced by $(\nabla_\Gamma\bu_T^n)\bu_T^{n-1}$. This approach has a straightforward extension to higher order in $\Delta t$ schemes, cf. the discussion in Remark~\ref{remissues} below. The space discretization of  \eqref{NS_FD}--\eqref{e:nExt1} is presented in the next section.

\subsection{Space discretization method} \label{s:fulldisc}
%We now explain the spatial discretization of \eqref{NS_FD}--\eqref{e:nExt1}.
Consider  a fixed polygonal domain  $\Omega \subset \mathbb{R}^3$ that strictly contains $\Gamma(t)$ for all $t\in[0,T]$.
Let $\{\T_h\}_{h>0}$ be  a family of shape-regular consistent triangulations of $\Omega$, with $\max\limits_{K\in\T_h}\mbox{diam}(K) \le h$. Corresponding to the bulk triangulation we define a standard finite element space of piecewise polynomial continuous functions of a fixed degree $k\ge1$:
\begin{equation} \label{eq:Vh}
	V_{h,k}=\{v_h\in C(\Omega)\,:\, v_h\in P_k(K),~~ \forall K\in \mathcal{T}_h\}.
\end{equation}
The bulk velocity and pressure finite element spaces are standard Taylor--Hood spaces:
\begin{equation} \label{TaylorHood}
	\bU_h = (V_{h,m+1})^3, \quad Q_h =  V_{h,m},\quad\text{with}~ m\ge1. %,\quad\text{for}~k\ge1.
\end{equation}
%\todo{AR: somewhere a comment of h.o. Taylor--Hood}
We want to avoid the assumption  that all $\Gamma(t_n)$ are given in an explicit parametric form  or that the exact level set functions $\phi(t_n, \cdot)$ are known. If, for example, \eqref{NSalt} is  part of a system, where tangential surface motions are coupled to the normal ones, then finding $\Gamma(t_n)$ is  part of the problem, and  knowledge of only a (finite element) approximation to $\phi(t_n, \cdot)$ would be a more realistic  assumption. In turn, this results in approximation of all geometric quantities involved in \eqref{NSalt}. This issue of geometry approximation is explained in section~\ref{s:geom} below.

If not specified otherwise, all constants $C$, $c$, $c_0$, $c_1$, etc. appearing later in the text are  generic positive constants which are \emph{independent} of $h$, $\Delta t$, other discretization parameters,  time instance $t_n$, and the position of $\Gamma$ in the background mesh, but \emph{may depend} on $w_N$, $\Gs$, $\bu$, and the shape regularity of $\T_h$.
In order to reduce the repeated use of such  constants,   we often write $x\lesssim y$ to state that the inequality  $x\le c y$ holds for quantities $x,y$ with such  generic  constant $c$. Similarly for $x\gtrsim y$, and $x\simeq y$ will mean that both $x\lesssim y$ and $x\gtrsim y$ hold.

%However, we shall continue to monitor the explicit dependence of the estimates on the (norm of the) normal surface velocity $\wn$.

%Moreover, an accurate numerical integration along   $\Gamma(t)$ is often non-feasible and is replaced by integration over some $\Gamma_h(t)$ approximating the true surface.

\subsubsection{Geometry approximation} \label{s:geom}
For any fixed $t\in[0,T]$,  $\phi_h(\cdot)= \phi_h(t,\cdot)$ is a given continuous piecewise polynomial approximation (with respect to $\T_h$) of $\phi(\cdot)=\phi(t,\cdot)$, which satisfies
\begin{equation}\label{phi_h}
	\|\phi-\phi_h\|_{L^\infty(\Omega)}+ h \|\nabla(\phi-\phi_h)\|_{L^\infty(\Omega)}\lesssim \,h^{q+1},
\end{equation}
with some $q\ge1$. For this estimate to hold, we assume that the level set function has the smoothness property $\phi(t, \cdot) \in C^{q+1}(\Omega)$.
Moreover, we assume that $|\nabla \phi_h| \ge C>0$ in  $\mathcal{O}(\Gamma(t))$, $t\in[0,T]$, and that $\phi_h$ is sufficiently regular in time such that with $\phi_h^n(\bx) = \phi_h(t_n, \bx ),~n=0,\dots,N$, there holds
\begin{subequations}\label{phi_hb}
	\begin{align}\label{phi_hba}
		\| \phi_h^{n-1}-\phi_h^n\|_{L^\infty(\Omega)} & \lesssim\,\Delta t \|\wn\|_{L^\infty(I_n \times \Omega)}, \\
		\label{phi_hbb}
		\|\nabla \phi_h^{n-1}-\nabla \phi_h^n\|_{L^\infty(\Omega)} & \lesssim\,\Delta t\left( \|\wn\|_{L^\infty(I_n \times \Omega)} +  \|\nabla \wn\|_{L^\infty(I_n \times \Omega)}\right), \text{ for } n=1,\dots,N.
	\end{align}
\end{subequations}
We define the discrete surfaces $\G{}{h}\approx \Gamma$ as the zero level of $\phi_h$:
\begin{equation} \label{discrGamma}
\G{}{h}(t):=\{\bx\in\rr^3\,:\,\phi_h(t,\bx)=0\}.
\end{equation}
From the property \eqref{phi_h} it follows that the Lipschitz surface $\G{}{h}$ is an approximation to $\Gamma$ with
\begin{equation}\label{eq:dist}
	\operatorname{dist}(\G{n}{h},\G{}{}) = \max_{x\in\G{}{h}} |\phi(\bx)|
	= \max_{x\in\G{}{h}} |\phi(\bx) - \phi_h(\bx)| \leq \Vert \phi - \phi_h \Vert_{\infty,\Omega} \lesssim h^{q+1}.
\end{equation}
Furthermore, it also follows that the vector field $\bn_h=\nabla\phi_h/|\nabla\phi_h|$ ($\bn_h$ is the
normal to $\G{}{h}$), and the extended normal vector to $\G{}{}$   satisfy
\begin{equation}\label{eq:normals}
	\vert \bn_h(\bx) - \bn(\bx) \vert \leq c | \nabla \phi_h(\bx) - \nabla \phi(\bx) | \lesssim h^q, \quad\text{in}~\mathcal{O}(\Gamma(t)).
\end{equation}
On $\Gamma_h$ we have the tangential projection operator $\bP_h= \bI - \bn_h \bn_h^T$.
Besides the geometric quantity $\bn_h$ we also need a discrete approximation of the Weingarten mapping $\bH_h \approx \bH$.
We assume that this approximation is of the form
%{\bf AR: here I made a change; modified assumption used to (hopefully) improve the $h^{q-1}$ bounds in\eqref{RESI5} and \eqref{est:consist2} to $h^q$; this assumption is consistent with how $\bH_h$ is computed in practice, e.g.,  the (surface) gradient applied to a projection (on the FE space) of discrete normal $\bn_h$}\\
\begin{equation} \label{discrWeingarten}
  \bH_h= \nabla_{\Gamma_h} \overline{\bn}_h, \quad \text{with}~\overline\bn_h \in H^1(\Gamma_h)^3\quad \text{that satisfies}~ \vert \overline{\bn}_h(\bx) - \bn(\bx) \vert \lesssim h^{q}, \quad \bx \in \Gamma_h.
\end{equation}
Here $\nabla_{\Gamma_h}$ is defined (a.e. on $\Gamma_h$)  analogous to $\nabla_\Gamma$.
The normal approximation $\normalbar$ may be chosen as a suitable interpolation operator in a finite element space applied to $\bn_h$, e.g., $\normalbar= I_h \bn_h$, with $I_h$ the (componentwise) Oswald averaging operator that maps into the finite element space $(V_{h,q-1})^3$.

We emphasize that all bounds in \eqref{phi_h}--\eqref{discrWeingarten} are uniform in $t$ (as well as in $h$, $\Delta t$, $n$ and position of $\Gamma$ or $\Gamma_h$ in the mesh).

\begin{remark} \label{remiso} \rm
In the method that we present below integrals over $\Gamma_h$ occur. We \emph{assume that these integrals  can be computed accurately}. In practice, this is straightforward for piecewise linear $\phi_h(\cdot,t)$.  The higher order case $q>1$ is more involved and requires special approaches for the construction of quadrature rules or the use of an parametric FEM  technique~\cite{grande2018analysis,fries2015,lehrenfeld2016high,muller2013highly,olshanskii2016numerical,saye2015hoquad,sudhakar2013quadrature}.
\end{remark}

\subsubsection{Fully discrete method} \label{sectfullydiscrete}
For computational efficiency reasons, we use an extension not in the given ($h$ and $\Delta t$-independent) neighborhood $\mathcal{O}(\Gamma(t_n))$ of $\Gamma(t_n)$ but in  a smaller ($\Delta t$-dependent) \emph{narrow band} around   $\G{n}{h}=\Gamma_h(t_n)$. This  narrow band consists   of all tetrahedra within a $\delta_n$ distance from the surface, with $\delta_n\simeq\Delta t$.
More precisely, we define the mesh-dependent narrow bands
\begin{align}
	\begin{split}\label{eq:def_narrow_band}
	U_{\delta}(\G{n}{h}) &:= \left\{ \bx \in \Omega \,:\mbox{dist}(\bx,\G{n}{h}) \le\delta_n\right\}, \\
	\mathcal{O}_{\delta}(\G{n}{h})& := {\bigcup}\left\{\overline{K}\,:\,  K\in\mathcal{T}_h~\text{and}~\mbox{dist}(\bx,\G{n}{h}) \le\delta_n  \text{ for some } \bx \in K\right\} \supset
	U_{\delta}(\G{n}{h}).
	\end{split}
\end{align}
We also need a subdomain of $\mathcal{O}_{\delta}(\G{n}{h})$   consisting of tetrahedra intersected by $\G{n}{h}$,
\begin{equation*}
	\omega^{n}_{\Gamma}:=
	{\bigcup}\left\{ \overline{K}\in\mathcal{T}_h\,:\, K\cap\G{n}{h}\neq\emptyset\right\}.
\end{equation*}
Note that the subdomains $\mathcal{O}_{\delta}(\G{n}{h})$ and $\omega_\Gamma^n$ consist of unions of tetrahedra $K \in \mathcal{T}_h$.
The finite element spaces for velocity and pressure are \emph{restrictions to these narrow bands} $\mathcal{O}_{\delta}(\G{n}{h})$ and $\omega_\Gamma^n$ \emph{of the time-independent bulk spaces} $\bU_h$ and $Q_h$:
\begin{equation}\label{eq:testtrial}
	\bU_h^n: =\{\, {\bv}|_{\mathcal{O}_{\delta}(\G{n}{h})}~|~\bv \in\bU_h\, \},\quad Q_h^n:=\{\, {q}|_{\omega_\Gamma^n}~|~q \in Q_h\, \}.
\end{equation}
We also use the notation $V_{h,m}^n:=\{\, v|_{\mathcal{O}_{\delta}(\G{n}{h})}~|~v \in V_{h,m}\, \}$.
In the derivation of a finite element formulation based on the discrete-in-time system \eqref{NS_FD}-\eqref{e:nExt1}, we need to address three important aspects: tangentiality of $\bu_T^n$, extension of the velocity along normal directions as in  \eqref{e:nExt1} and a handling of the inertia term.
First, we relax the condition for the solution to be tangential to $\Gamma(t_n)$ to allow for $\bU_h^n$ as trial and test velocity space. The tangentiality condition is weakly enforced using a penalty  approach, which is often used in finite element methods for vector-valued surface PDEs~ \cite{hansbo2020analysis,jankuhn2019higher,olshanskii2018finite,olshanskii2019penalty}.
The constraint in \eqref{e:nExt1} is also relaxed. For this we use a penalty (or stabilization) approach that is standard in trace finite element method and based on adding a volume normal derivative term to the discrete bilinear form. The treatment of inertia also follows an established approach. For this we rewrite the corresponding trilinear form, where we use integrals over the exact surface $\Gamma=\Gamma(t_n)$ and $\bu_T \cdot \nabla_\Gamma \bv_T =(\nabla_\Gamma \bv_T)\bu_T$:
\[
  \int_{\Gamma} (\bu_T \cdot \nabla_\Gamma \bu_T) \cdot \bv_T \, ds = \tfrac12 \int_{\Gamma}( \bu_T \cdot \nabla_\Gamma \bu_T) \cdot \bv_T - (\bu_T \cdot \nabla_\Gamma \bv_T) \cdot \bu_T\, ds - \tfrac12 \int_\Gamma     \divG \bu_T (\bu_T \cdot \bv_T)\, ds.
\]
Using $\divG \bu_T=f$ on $\Gamma$,  the second term on the right-hand side becomes linear. The first term is linearized using the $\bu_T$ approximation from the previous time step   leading to an antisymmetric bilinear form, which is convenient for the stability analysis of the discretization.

Putting these components together leads to  the following \emph{fully discrete problem} (one time step), where we use (sufficiently accurate) extensions $w^e_{N}$, $f^e$, $\mathbf{f}^e$ of the data $w_{N}$, $f$, $\mathbf{f}$: For  given $\bu_h^{n-1} \in \bU_h^{n-1}$ find  $\bu_h^n\in \bU_h^n$, $p_h^n\in Q_h^n$,   satisfying
\begin{align}
	& \int_{\G{n}{h}}\left(\frac{\bu^n_h-\bu^{n-1}_h}{\Delta t}  + w_N^{e,n} \bH_h \bu_h^n\right)\cdot\bP_h\bv_h\,ds_h    \nonumber \\
     &+ \tfrac12 \int_{\G{n}{h}}( \bu_h^{n-1} \cdot \nabla_{\Gamma_h}\bP_h \bu_h^n) \cdot \bv_h - (\bu_h^{n-1} \cdot \nabla_{\Gamma_h} \bP_h\bv_h) \cdot \bu_h^n\, ds_h - \tfrac12 \int_{\G{n}{h}} f^{e,n} \bu_h^n \cdot \bP_h \bv_h \, ds_h \nonumber \\
	 & + 2\mu \int_{\G{n}{h}} E_{s,h}(\bP_h\bu_h^n):E_{s,h}(\bP_h\bv_h)\,ds_h
	+ \underbrace{\tau\int_{\G{n}{h}}  (\tbn_h\cdot\bu^n_h)(\tbn_h\cdot\bv_h)\,ds_h}_{\text{penalty for } \bn\cdot\bu^n=0}   \label{e:FEM} \\
	& + \int_{\G{n}{h}} \nabla_{\Gamma_h} p_h^n \cdot \bv_h\,ds_h+ \underbrace{\rho_u\int_{\mathcal{O}_{\delta}(\G{n}{h})}(\bn_h\cdot\nabla \bu^{n}_h)(\bn_h\cdot\nabla \bv_h)\,\diff{\vect x}}_{\text{velocity stabilization and extension}} = \int_{\G{n}{h}} \blf^{e,n}\cdot\bv_h\,ds_h ~ \forall\, \bv_h\in \bU_h^n,  \nonumber \\
	& - \int_{\G{n}{h}} \nabla_{\Gamma_h} q_h \cdot \bu_h^n\,ds_h
	+ \underbrace{\rho_p  \int_{\omega^{n}_{\Gamma}} (\bn_h \cdot \nabla p_h^n)  (\bn_h \cdot\nabla q_h) \,  \diff{\vect x}}_{\footnotesize \text{pressure stabilization}} =\int_{\G{n}{h}}f^{e,n} q_h\,ds_h \quad \forall ~q_h\in Q_h^n. \label{e:FEMa}
\end{align}
Here we mimic the notation used in the PDE system on $\Gamma$. For example, $\nabla_{\Gamma_h} \bv_h = \bP_h \nabla \bv_h \bP_h$, $E_{s,h}(\bw_h)= \tfrac12 (\nabla_{\Gamma_h} \bw_h + \nabla_{\Gamma_h} \bw_h^T)$.
The additional term $\tau\int_{\G{n}{h}}  (\tbn_h\cdot\bu^n_h)(\tbn_h\cdot\bv_h)\,ds$ with penalty parameter $\tau>0$ is included to weakly enforce the tangentiality condition $ \tbn_h\cdot\bu^n_h=0$ on $\G{n}{h}$. In this term we use an ``improved normal'', denoted by $\tbn_h$, which has one order better accuracy (as approximation of $\bn$) than the discrete surface normal $\bn_h$. We assume (compare to \eqref{eq:normals}):
\begin{equation} \label{eq:normalimproved}
  |\tbn_h(\bx)- \bn(t_n,\bx)| \lesssim h^{q+1}, \quad \bx \in \G{n}{h}.
\end{equation}
From the literature it is known that such a more accurate normal in this penalty term is needed for optimal order discretization errors, cf. \cite{hansbo2020analysis,Jankuhn2020}. In section~\ref{s:Numerics} we explain how $\tbn_h$ is determined in our implementation of the method.
The volumetric term $\int_{\mathcal{O}_{\delta}(\G{n}{h})}(\bn_h\cdot\nabla \bu^{n}_h)(\bn_h\cdot\nabla \bv_h)\,\diff{\vect x}$, scaled by the parameter $\rho_u$, plays a twofold role. Firstly, due to this term instabilities caused by ``small cuts'' are damped, resulting in  satisfactory conditioning of the stiffness matrix for velocity, cf. e.g.~\cite{burmanembedded}.
Secondly, this  term weakly enforces the extension condition \eqref{e:nExt1} with
$\mathcal{O}(\Gamma(t_n))$ replaced by $\mathcal{O}_{\delta}(\G{n}{h})$. %Consequently, at time $t_n$ a well-posed algebraic system is solved for all discrete velocity degrees of freedom in the neighborhood $\mathcal{O}_{\delta}(\G{n}{h})$.
 The  volumetric term $\rho_p  \int_{\omega^{n}_{\Gamma}} (\bn_h \cdot \nabla p_h^n)  (\bn_h \cdot\nabla q_h) \,  \diff{\vect x}$  is added for the purpose of numerical  stabilization of pressure, both with respect to finite element (LBB) stability and conditioning of the resulting matrix, cf. \cite{olshanskii2021inf}.  Due to these stabilizations the algebraic system (in each time step) is well-posed and has conditioning properties comparable to those of a discretized linearized Navier-Stokes system in  Euclidean domains.

The formulation \eqref{e:FEM}--\eqref{e:FEMa} is \emph{consistent} up to geometric errors: If $\G{n}{h}$ is replaced by $\Gamma^n=\Gamma(t_n)$ and all geometric quantities by the corresponding exact ones (e.g., $\bH_h$ by $\bH$), then the equations in \eqref{e:FEM}--\eqref{e:FEMa} are satisfied with $\bu_h^n$, $p^n_h$, $\bu_h^{n-1}$ replaced by the solution $\bu^n$, $p^n$ and $\bu^{n-1}$ of \eqref{NS_FD}, extended along normal directions.

%For convenience we assume that the data extensions $w^e_{N}$, $f^e$, $\mathbf{f}^e$ are constant extensions along the normals $\bn$. In practice other extensions may be used..

\subsection{Discussion of the method} In addition to the mesh size parameter $h$ and time step $\Delta t$, the finite element method involves several other discretization parameters summarized below.
\begin{itemize}
    \item $\rho_u$: a normal stabilization and extension parameter for velocity;
	\item $\rho_p$: a normal stabilization parameter for pressure;
	\item $\tau$~: a penalty parameter for $\bn_h\cdot\bu_h=0$ constraint;
	\item $\delta_n$: a narrow band parameter at time $t_n$.
\end{itemize}
We assume that these parameters satisfy the  conditions
\begin{subequations}\label{Conds}
\begin{align}
	h & \le c_1,~~\Delta t\le c_2, \label{CondA}\\
	 c_\rho h^{-1} & \leq \rho_u \lesssim h^{-1},\label{CondD}\\
	 \rho_p & \simeq h,	\label{CondE} \\
	  \tau & \simeq  h^{-2}, \label{CondTau}\\
	 \Delta t & \geq 2 \,\tau^{-1},\label{CondC}\\
	c_\delta\Delta t & \le \delta_n \lesssim h,\label{CondB}
\end{align}
\end{subequations}
where $c_1$, $c_2$ are sufficiently small $O(1)$ and $c_\delta$, $c_\rho$ are sufficiently large $O(1)$ constants independent of the parameters, time,  and position of $\Gamma_h$ in the mesh. The parameter conditions \eqref{CondD} and \eqref{CondE} are known from the literature on trace finite element methods.  The condition ``$c_\rho$ sufficiently large'' is  needed to obtain  sufficient control of a velocity extension in the narrow band (it is used to prove a key estimate \eqref{fund2simple}). Parameter $\tau$ as in \eqref{CondTau} guarantees accurate enough fulfillment of the tangentiality condition for the discrete velocity. Taking $\tau$  even larger may lead to a 'locking' phenomenon. Eq.~\eqref{CondC} is a technical condition that we need to prove the  stability estimate in Theorem~\ref{Th1}.  The upper bound in~\eqref{CondB} on the narrow band parameter $\delta_n$ keeps the complexity of the method optimal, i.e. the number of active degrees of freedom is $O(h^{-2})$ on each time step. Concerning the lower bound on  $\delta_n$ we note the following.
In a time step from $t_{n-1}$ to $t_n$, the surface $\Gamma(t)$ can move at most $\Delta t\sup_{t\in I_n}\| w_N\|_{L^\infty(\Gamma(t))}$ distance in normal direction. Thanks to assumption \eqref{phi_hba} the maximum  distance from $\G{n}{h}$ to $\G{n-1}{h}$ is also proportional to $\Delta t\sup_{t\in I_n}\| w_N\|_{L^\infty(\Gamma(t))}$.
Therefore, $c_\delta$ in \eqref{CondB} can be taken  such that
\begin{equation}\label{cond1}
	%\G{n}{h}\subset
	\omega^{n}_{\Gamma}\subset\mathcal{O}_{\delta}(\G{n-1}{h}).
\end{equation}
%To see this, one applies \eqref{e:delta} to determine $\delta_{n-1}$, which in turn defines ${\mathcal{O}}_h(\G{n-1}{h})$.
This condition is the discrete analog of \eqref{ass1} and  is essential for the well-posedness of the finite element problem. The assumptions \eqref{CondTau}--\eqref{CondB} can be satisfied if we take  $\Delta t$ and $h$ such that the scaling conditions
\begin{equation} \label{condmain}
h^2\lesssim \Delta t \lesssim h
\end{equation}
hold. These scalings are reasonable. Consider Taylor--Hood elements with $m=1$. For optimal  $O(h^2)$-convergence in the energy norm of \eqref{e:FEM} one needs $\Delta t \simeq h^2$. For the more time-accurate BDF2 scheme (this is our practical choice, cf. Remark~\ref{BDF2}) $\Delta t \simeq h$ leads to $O(h^2)$-convergence in the energy norm, and $\Delta t \simeq h^{3/2}$ is the choice consistent with the best possible $O(h^3)$-order in the velocity $L^2$-norm.
In the remarks below we briefly address a few further aspects of the method.

\begin{remark} \label{remissues}  \rm In practice we typically do not use a level set function that has the signed distance property. Then the extension does not satisfy \eqref{e:nExt1}, but is approximately constant along the normals to the level lines and  instead of $\bP \pc \bu_T=\bP \frac{\partial \bu_T}{\partial t}$, cf. \eqref{tderivative}, one uses the general relation
\[
\bP \pc \bu_T = \bP\left( \frac{\partial \bu_T^e}{\partial t} + (\nabla \bu_T^e) \bw_N\right),
\]
as basis for the discretization. Hence, for the case of a general level set function we include  in \eqref{e:FEM}  the term $\int_{\Gamma_h^n} (\bw_N \cdot \nabla \bP_h \bu_h^n)\cdot \bP_h \bv_h \, ds_h$.  Apart from this, the discretization method stays the same.
\end{remark}
\begin{remark}\rm \label{BDF2}
The BDF2 variant is very similar to the method introduced above. The time difference $\bu_h^n - \bu_h^{n-1}$ in the first line in \eqref{e:FEM} is then replaced by $ \tfrac32 \bu_h^n - 2\bu_h^{n-1}+ \tfrac12\bu_h^{n-2} $ and in the linearization of the inertia term we use $2\bu_h^{n-1}- \bu_h^{n-2}$ instead of $\bu_h^{n-1}$. All other terms in \eqref{e:FEM}--\eqref{e:FEMa} remain the same. For BDF2, in addition to \eqref{cond1} one has to guarantee $\omega^{n}_{\Gamma}\subset\mathcal{O}_{\delta}(\G{n-2}{h})$, which means that the narrow bands $\mathcal{O}_{\delta}(\G{n}{h})$ have to be taken thicker than in the BDF1 method.
\end{remark}
\begin{remark} \label{remGrad}\rm
In the surface gradient operators $\nabla_{\Gamma_h}$ in the inertia term and the surface rate-of-strain tensor $E_{s,h}(\cdot)$ we use projected vector fields $\bP_h\bv$. Differentiating projector $\bP_h$ (which in general is discontinuous across element faces) we loose $H^1$ conformity, cf. \cite{Jankuhn2020}. Based on the formula $\nabla_\Gamma \bv_T = \nabla_\Gamma \bv - v_N \bH$, in order to avoid differentiation of $\bP_h$, in the implementation of the method we replace $\nabla_{\Gamma_h} \bP_h \bv_h$ by $\nabla_{\Gamma_h}  \bv_h - (\bn_h \cdot \bv_h) \bH_h$.
\end{remark}
\begin{remark}\rm
For convenience, in the error analysis below we assume that the data extensions $w^e_{N}$, $f^e$, $\mathbf{f}^e$ are constant extensions along the normals $\bn$. In practice one typically uses other (sufficiently accurate) extensions.
\end{remark}
\section{Error analysis} \label{s:FEM}
We continue with an error analysis of the finite element scheme~\eqref{e:FEM}--\eqref{e:FEMa}.
We assume the solution of \eqref{NSalt} is sufficiently smooth, at least  $\bu\in W^{m+2,\infty}(\Gs)^3$, $p \in W^{m+1,\infty}(\Gs)$.
% A priori error bounds in an energy norm are  derived in subsection \ref{sec:aprioriest}, based on stability and consistency analyses (subsections
%\ref{s_stab} and \ref{s:consistency}).
In the remainder we assume that the parameter conditions \eqref{CondA}--\eqref{CondB} are satisfied.
As common in analysis  of  incompressible fluid problems, we  assume a homogeneous divergence condition in \eqref{NSalt}:
\[
f=0.
\]
See the Remark~\ref{remf}.
We use the notations $(\cdot, \cdot)_S$ and $\|\cdot\|_S$ to denote the $L^2$-scalar product and $L^2$-norm over a surface or volumetric domain $S$.
To represent the discrete problem in a compact form, we introduce
\begin{equation} \label{e:def:an}
  \begin{split}
    a_n(\bz;\bu,\bv) := & \widehat a_n(\bu,\bv)+c_n(\bz;\bu,\bv),\quad \text{with}\\
   \widehat a_n(\bu,\bv) := & \int_{\G{n}{h}}w_N^{e,n} \bv\cdot\bH_h \bu\,ds_h +2\mu\int_{\G{n}{h}}(E_{s,h}(\bP_h\bu):(E_{s,h}(\bP_h\bv))\, ds_h\\ & \quad + \tau\int_{\G{n}{h}}  (\tbn_h\cdot\bu)(\tbn_h\cdot\bv)\,ds_h +\rho_u\int_{\mathcal{O}_\delta(\G{n}{h})}(\bn_h\cdot\nabla \bu)(\bn_h\cdot\nabla \bv) d\,\bx\\
     c_n(\bz;\bu,\bv) := & \tfrac12\int_{\G{n}{h}}  (\bz\cdot\nabla_{\Gamma_h} \bP_h \bu)\cdot \bv - (\bz\cdot\nabla_{\Gamma_h} \bP_h  \bv)\cdot \bu   \, ds_h \\
    b_n(p,\bv) := &\int_{\G{n}{h}}\nabla_{\Gamma_h} p \cdot \bv \,ds_h.
  \end{split}
\end{equation}
Using this,  the discrete problem can be rewritten as follows (recall that $f=0$): given $\bu_h^{n-1} \in \bU_h^{n-1}$  %and ${\bc}^{n-1}_h={\bu}^{n-1}_h+ I_h(w_N^e(t_n)\bn)$
find  $\bu_h^n\in \bU_h^n$, $p_h^n\in Q_h^n$ such that for all $(\bv_h,q_h) \in \bU_h^n \times Q_h^n$:
\begin{equation} \label{discrcompact}
 \begin{split}
  (\bu_h^n, \bP_h \bv_h)_{\G{n}{h}} + \Delta t\, a_n(\bu_h^{n-1};\bu_h^n, \bv_h) + \Delta t\, b_n(p_h^n,\bv_h) &= ( \bu_h^{n-1} , \bP_h \bv_h)_{\G{n}{h}}  +\Delta t (\blf^{e,n},\bP_h \bv_h)_{\G{n}{h}} \\ b_n(q_h,\bu^n_h) & =\rho_p ( \bn_h \cdot \nabla p_h^n ,  \bn_h \cdot\nabla q_h)_{\omega_\Gamma^{n}}  .
 \end{split}
\end{equation}
We   introduce the following natural energy norm
\begin{equation} \label{defnorms}
		\enormu{\bv}^2  :=\tfrac12\|\bv\|^2_{\G{n}{h}}+2\mu\|E_{s,h}(\bP_h \bv)\|^2_{\G{n}{h}}+\tfrac\tau2\|\tbn_h\cdot\bv\|^2_{\G{n}{h}}+\rho_u\|\bn_h \cdot \nabla \bv\|_{\mathcal{O}_\delta(\G{n}{h})}^2 . %\quad \bv\in H^1(\mathcal{O}_\delta(\G{n}{h}))^3
%		\\
%		 \enormp{p}^2 & := \Vert p \Vert_{\G{n}{h}}^2 + \rho_p \Vert \bn_h \cdot \nabla p \Vert_{\omega^n_{\Gamma}}^2. %\quad p\in H^1(\mathcal{O}_\delta(\G{n}{h})).
\end{equation}
The norm that we use for the pressure will be introduced later in \eqref{defpnorm}.

\begin{remark} \label{RemAnalysis} \rm
 We outline the structure of the error analysis. We first collect some  preliminary  results and prove a few helpful estimates, with a particular emphasis on having them uniform in discretization parameters, time and the surface position in the bulk mesh. Using these results it is easy to derive a suitable coercivity and continuity estimates for the trilinear form $a_n(\cdot;\cdot,\cdot)$ used in \eqref{discrcompact}.
Based on rather straightforward arguments we derive  in Theorem~\ref{Th1} a  stability result for the discrete velocity solution $\bu_h^n$ of  \eqref{discrcompact} in a natural energy norm. The analysis continues with addressing the more delicate question of the finite element pressure stability. Here we introduce  new  arguments, namely the use of a specific (non-standard)   pressure norm and of a special discrete velocity function that gives control of the pressure via the discrete inf-sup stability condition, cf. Remark~\ref{remArgument}. Thus we obtain a uniform discrete stability estimate (Theorem~\ref{pstability}) in the parameter range \eqref{condmain}.
 In Section~\ref{s:consistency} we present a consistency analysis that  quantifies geometric errors resulting from the approximation of surface and geometric quantities (normal vectors, tangential projectors, curvatures) in the finite element formulation. In that analysis we use results already available in the literature on surface vector-Laplace and Stokes problems. Finally, in Section~\ref{sec:aprioriest} we apply a standard technique for proving discretization error bounds. The discretization error is split into an interpolation error and an error component that lies in the finite element space. Bounds for the latter can be derived using the discrete stability, consistency error bounds, and bounds for the linearization error. With respect to the interpolation error a key point is that in the Euclidean setting of this trace method we can use a (nodal) interpolation operator on a fixed  bulk mesh. Due to this the time differentiation and interpolation operator commute, which then leads to optimal bounds for the difference of interpolation errors at $t=t_n$ and $t=t_{n-1}$, cf. Lemma~\ref{l_interp}.
\end{remark}

\subsection{Preliminaries} \label{sec:volumecontrol}
We start with an equivalence result  between norms on $\Gamma$ and $\Gamma_h$, which is uniform in time and discretization parameters. We use a lifting $v^\ell$ of functions $v$ defined on $\Gamma_h^n$ to a neighborhood ${\mathcal O}(\Gamma_h^n)$, defined as follows:
\begin{equation} \label{deflifting}
\text{for}~~\by \in {\mathcal O}(\Gamma_h^n): \quad
v^\ell(\by):=v(\bx)~~\text{for}~\bx \in \Gamma_h^n~~\text{such that}~~\bp^n(\by)=\bp^n(\bx),
\end{equation}
where $\bp^n$ is the closest point projection on $\Gamma^n$. Differentiating  the identities $\bv^\ell(\by)=\bv^\ell(\bp^n(\by))$ one obtains the useful transformation relation
\begin{equation} \label{transform}
 \nabla \bv^\ell(\by)= \nabla \bv^\ell(\bp^n(\by))(\bP(\by)- d(\by)\bH(\by)), \quad \bv \in H^1(\Gamma_h^n)^3,
\end{equation}
with $d$ the signed distance function for $\Gamma^n$.
For vector-valued functions we define componentwise $H^1$ norms, namely
$
 \|\bv\|_{H^1(\Gamma(t_n))}^2  :=\sum_{i=1}^3\|v_i\|_{H^1(\Gamma(t_n))}^2$ for $\bv\in H^1(\Gamma(t_n))^3$, and similarly on $\Gamma_h^n$
 for $\bv\in H^1(\Gamma_h^n)^3$.
 % \|\bv\|_{H^1(\Gamma(t_n))}^2   &:= \int_{\Gamma(t_n)} \|\bv(s)\|_2^2 +\|\nabla \bv^e(s)\bP(s)\|_2^2\, ds, \quad \bv\in H^1(\Gamma(t_n))^3 \\
%\|\bv\|_{H^1(\Gamma_h^n)}^2 & := \int_{\Gamma_h^n} \|\bv(s)\|_2^2 +\|\nabla \bv^\ell(s)\bP_h(s)\|_2^2 \, ds_h, \quad \bv \in H^1(\Gamma_h^n)^3.
For any $v\in H^1(\Gamma_h^n)$, $\bv \in H^1(\Gamma_h^n)^3$  it holds
\begin{equation}
	\|v\|_{L^p(\Gamma_h^n)}\simeq \|v^\ell\|_{L^p(\Gamma(t_n))},~p\in[1,\infty], \quad
	\|\bv\|_{H^1(\Gamma_h^n)}\simeq \|\bv^\ell\|_{H^1(\Gamma(t_n))}. \label{eqv1}
\end{equation}
The first equivalence in \eqref{eqv1} is straightforward and follows from the uniform equivalence of the surface measures $ds$ (on $\Gamma(t_n)$) and $ds_h$ (on $\Gamma_h^n$). A proof of the second equivalence in \eqref{eqv1} is given in \cite[Lemma 5.14]{Jankuhn2020}.
We also need uniform interpolation and Korn-type inequalities:
\begin{align}
	\|v\|_{L^4(\Ghn)} & \lesssim \|v\|_{\Ghn}^\frac12 \|v\|_{H^1(\Ghn)}^\frac12 \quad \text{for all}~v \in H^1(\Ghn) \label{intestimate}
		\\
		\|\bv\|_{L^4(\Ghn)} & \lesssim \|\bv\|_{\Ghn}^\frac12 \|\bv\|_{H^1(\Ghn)}^\frac12 \quad \text{for all}~\bv \in H^1(\Gamma_h^n)^3 \label{intestimateV} \\
	\|\bv\|_{H^1(\Ghn)} & \lesssim \enormu{\bv} \quad\text{for all}~\bv \in\bU^n_h.  \label{Kornestimate}
\end{align}
The results \eqref{eqv1}, \eqref{intestimate} and \eqref{Kornestimate} are derived in Appendix, Section~\ref{AppSecA}. Componentwise application of \eqref{intestimate} yields \eqref{intestimateV}.
%where
%\[
%\cH_n:= \{\, \bv \in H^1\big(\mathcal{O}_\delta(\G{n}{h})\big)^3~|~\enormu{\bv} < \infty\,\}.
%\]

In  \cite{lehrenfeld2018stabilized}  a finite-difference in time -- trace finite element method in space   was analyzed, but for a \emph{scalar} parabolic type problem on an evolving surface. The following lemma collects some results from  that paper, useful for our error analysis.

\begin{lemma} \label{lemcrucial}
The following uniform  estimates hold
%  with $\tilde{u} = u \circ \Phi^{-1}$:
\begin{subequations}
  \begin{align}\label{fund1a}
\|v_h\|_{U_{\delta}(\G{n}{h})}^2 & \lesssim  \delta_n \|v_h\|_{\G{n}{h}}^2 + \delta_n^2 \|\bn_h \cdot \nabla v_h\|_{\mathcal{O}_\delta(\G{n}{h})}^2\quad \forall\,v_h \in V_{h,m}^{n},
 \\
    \label{fund1}
\|v_h\|_{\mathcal{O}_\delta(\G{n}{h})}^2 & \lesssim  (\delta_n + h) \|v_h\|_{\G{n}{h}}^2 + (\delta_n + h)^2 \|\bn_h \cdot \nabla v_h\|_{\mathcal{O}_\delta(\G{n}{h})}^2\quad \forall\,v_h \in V_{h,m}^{n},
\\
%\label{fund2}
%\|v_h\|_{\G{n}{h}}^2
%&\le
%(1+c \Delta t)
%\|v_h\|_{\G{n-1}{h}}^2 + c_0 \frac{\delta_{n-1}}{\delta_{n-1}+h} \|\bn_h^{n-1} \cdot \nabla v_h\|_{\mathcal{O}_\delta(\G{n-1}{h})}^2, \quad \forall\,v_h \in V_h^{n-1}.
\label{fund2simple}
\|v_h\|_{\G{n}{h}}^2
&\le
(1+c \Delta t)
\|v_h\|_{\G{n-1}{h}}^2 +  \tfrac12\rho_{u} \Delta t \|\bn_h^{n-1} \cdot \nabla v_h\|_{\mathcal{O}_\delta(\G{n-1}{h})}^2 \quad \forall\,v_h \in V_{h,m}^{n-1}.
\end{align}
\end{subequations}
\end{lemma}
The results in \eqref{fund1a}--\eqref{fund1}  give control of the $L^2$-norm of a finite element function in a narrow band volume based on
a combination of the $L^2$-norm on the surface and the normal derivative in the volume that is provided by the stabilization. Eq.~\eqref{fund2simple} provides a bound for the trace of a function on $\G{n}{h}$ through its trace on $\G{n-1}{h}$ and the volume normal derivative. We also need a slightly modified version of the bound in \eqref{fund2simple}, given in the following lemma.
\begin{lemma} \label{lem2b} The following uniform estimate holds for all $v_h \in V_{h,m}^{n-1}$,
	\begin{equation} \label{fund3}
		\|v_h\|_{\G{n}{h}}^2
		\lesssim
		\|v_h\|_{\G{n-1}{h}}^2 + \delta_{n-1} \|\bn_h^{n-1} \cdot \nabla v_h\|_{\mathcal{O}_\delta(\G{n-1}{h})}^2.
	\end{equation}
\end{lemma}
 A proof is given in Appendix~\ref{Appsec2}. It uses arguments similar to those found in~\cite{lehrenfeld2018stabilized}.

In the analysis of two-dimensional Navier-Stokes equations one typically needs bounds for $\|\bu\|_{L^4}$ to control the quadratic nonlinearity. We will derive such a bound in Lemma~\ref{lemmaL4vect}. As a preparation, in the next lemma we give an analog of the result \eqref{fund3}, with the norm  $\|\cdot\|_{L^2}$ replaced by $\|\cdot\|_{L^4}$.
%In the proof we use the previous Lemma~\ref{lem2b}.
\begin{lemma} \label{lemmaL4}
The following uniform estimate holds
 \begin{equation} \label{A3}
	\|v_h\|_{L^4(\Ghn)}^2 \lesssim \|v_h\|_{L^4(\G{n-1}{h})}^2 +  \|\bn_h\cdot\nabla v_h\|_{\mathcal{O}_\delta(\G{n-1}{h})}^2, \quad \text{for all}~ v_h \in V^{n-1}_{h,m}.
	\end{equation}
\end{lemma}
\begin{proof}
We use the notation $\mathcal{O}_\delta=\mathcal{O}_\delta(\G{n-1}{h})$. First we consider an inverse estimate for $\|\bn_h \cdot \nabla v_h\|_{L^4(\mathcal{O}_\delta)}$. %Let $I_q\bn$ be a finite element interpolant of $\bn$ with $\|I_q \bn- \bn\|_{L^\infty(\mathcal{O}_\delta)} \lesssim h^q$.
 Using $\bn_h=|\nabla \phi_h|^{-1}\nabla \phi_h$, where $\phi_h$ is a finite element function with $|\nabla \phi_h|\simeq 1$ in $\mathcal{O}_\delta$, and finite element inverse estimates we get
\begin{equation} \label{A1}
 \begin{split}
	 \|\bn_h \cdot \nabla v_h\|_{L^4(\mathcal{O}_\delta)}&\le\||\nabla \phi_h|^{-1}\|_{L^\infty(\mathcal{O}_\delta)}\|\nabla \phi_h \cdot \nabla v_h\|_{L^4(\mathcal{O}_\delta)} \lesssim \|\nabla \phi_h \cdot \nabla v_h\|_{L^4(\mathcal{O}_\delta)}\\
	 &\lesssim h^{-\frac34}\|\nabla \phi_h \cdot \nabla v_h\|_{L^2(\mathcal{O}_\delta)}
	 \lesssim h^{-\frac34}\||\nabla \phi_h|\|_{L^\infty(\mathcal{O}_\delta)} \|\bn_h \cdot \nabla v_h\|_{L^2(\mathcal{O}_\delta)}\\
	  &\lesssim h^{-\frac34}\|\bn_h \cdot \nabla v_h\|_{L^2(\mathcal{O}_\delta)}.
\end{split}
\end{equation}
Using this and \eqref{fund1} applied to $v_h^2$ we get
\begin{align*}
 \|v_h\|_{L^4(\mathcal{O}_\delta)}^2 &=\|v_h^2\|_{L^2(\mathcal{O}_\delta)} \lesssim (\delta_{n-1}+h)^\frac12 \|v_h^2\|_{L^2(\G{n-1}{h})} + (\delta_{n-1}+ h) \|\bn_h \cdot \nabla v_h^2\|_{L^2(\mathcal{O}_\delta)} \\
  & \lesssim (\delta_{n-1}+h)^\frac12\|v_h\|_{L^4(\G{n-1}{h})}^2 +(\delta_{n-1}+ h)\|v_h\|_{L^4(\mathcal{O}_\delta)} \|\bn_h \cdot \nabla v_h\|_{L^4(\mathcal{O}_\delta)} \\
  & \leq c_0(\delta_{n-1}+h)^\frac12\|v_h\|_{L^4(\G{n-1}{h})}^2  + \frac12\|v_h\|_{L^4(\mathcal{O}_\delta)}^2 +c_1 h^{-\frac32}(\delta_{n-1} +h)^2 \|\bn_h \cdot \nabla v_h\|_{L^2(\mathcal{O}_\delta)}^2.
\end{align*}
Shifting the term $\|v_h\|_{L^4(\mathcal{O}_\delta)}^2$ to the left-hand side we obtain
\begin{equation} \label{A2}
 \|v_h\|_{L^4(\mathcal{O}_\delta)}^2 \lesssim (\delta_{n-1}+h)^\frac12\|v_h\|_{L^4(\G{n-1}{h})}^2 +h^{-\frac32}(\delta_{n-1} +h)^2 \|\bn_h \cdot \nabla v_h\|_{L^2(\mathcal{O}_\delta)}^2.
\end{equation}
We  use \eqref{fund3} with $v_h^2$, which is a finite element function of degree $2m$ and so \eqref{fund3} applies:
\begin{equation}\label{aux654A}
\|v_h\|_{L^4(\Ghn)}^2 = \|v_h^2\|_{L^2(\Ghn)}
\lesssim \|v_h^2\|_{L^2(\Gamma_h^{n-1})} +\delta_{n-1}^{\frac12}\|\bn_h\cdot\nabla v_h^2\|_{L^2(\mathcal{O}_\delta)}.
\end{equation}
From the chain rule and the result \eqref{A1} above we obtain
\begin{align*}
 \|\bn_h\cdot\nabla v_h^2\|_{L^2(\mathcal{O}_\delta)} & \le \|v_h\|_{L^4(\mathcal{O}_\delta)}\|\bn_h\cdot\nabla v_h\|_{L^4(\mathcal{O}_\delta)}
\lesssim  h^{-\frac34}\|v_h\|_{L^4(\mathcal{O}_\delta)} \|\bn_h\cdot\nabla v_h\|_{L^2(\mathcal{O}_\delta)}
\end{align*}
Substituting this in \eqref{aux654A} and using \eqref{A2} we get
\begin{align*}
 \|v_h\|_{L^4(\Ghn)}^2 & \lesssim \|v_h\|_{L^4(\Gamma_h^{n-1})}^2 +\delta_{n-1}^\frac12 h^{-\frac34}\|v_h\|_{L^4(\mathcal{O}_\delta)} \|\bn_h\cdot\nabla v_h\|_{L^2(\mathcal{O}_\delta)} \\
 & \lesssim \|v_h\|_{L^4(\Gamma_h^{n-1})}^2 +  \delta_{n-1} h^{-\frac32} \|v_h\|_{L^4(\mathcal{O}_\delta)}^2 + \|\bn_h\cdot\nabla v_h\|_{L^2(\mathcal{O}_\delta)}^2 \\
 & \lesssim \big(1+\delta_{n-1}h^{-\frac32}(\delta_{n-1}+h)^\frac12 \big)\|v_h\|_{L^4(\Gamma_h^{n-1})}^2+ \big(\delta_{n-1}h^{-3} (\delta_{n-1}+h)^2 +1\big) \|\bn_h\cdot\nabla v_h\|_{L^2(\mathcal{O}_\delta)}^2.
\end{align*}
Due to \eqref{CondB} (i.e. $\delta_{n-1} \lesssim h$) we get  $\delta_{n-1}h^{-\frac32}(\delta_{n-1}+h)^\frac12 \lesssim 1$ and $\delta_{n-1}h^{-3} (\delta_{n-1}+h)^2  \lesssim 1$, which completes the proof.
\end{proof}

Lemma~\ref{lemmaL4} is used in the proof of the next lemma.
\begin{lemma} \label{lemmaL4vect}
 The following uniform estimate holds
 \begin{equation} \label{A4}
  \|\bv_h\|_{L^4(\Gamma_h^n)}^2 \lesssim
%\|\bv_h\|_{\Gamma_h^{n-1}}^2+
\|\bv_h\|_{\Gamma_h^{n-1}}\enorm{\bv_h}_{U^{n-1}} + h\enorm{\bv_h}_{U^{n-1}}^2 \quad \text{for all}~\bv_h \in \bU_h^{n-1}.
 \end{equation}
\end{lemma}
\begin{proof}
Take $\bv_h \in \bU_h^{n-1}$. Component-wise application of \eqref{A3} yields
\begin{equation} \label{A5}
 \|\bv_h\|_{L^4(\Gamma_h^n)} \lesssim \|\bv_h\|_{L^4(\Gamma_h^{n-1})} +  \|\bn_h \cdot \nabla \bv_h\|_{\mathcal{O}_\delta (\Gamma_h^{n-1})} \lesssim   \|\bv_h\|_{L^4(\Gamma_h^{n-1})} + h^{\frac12}\enorm{\bv_h}_{U^{n-1}}.
\end{equation}
Using \eqref{intestimateV} and \eqref{Kornestimate} we get
\[
 \|\bv_h\|_{L^4(\Gamma_h^{n-1})} \lesssim \|\bv_h\|_{\Gamma_h^{n-1}}^\frac12 \enorm{\bv_h}_{U^{n-1}}^\frac12.
\]
Substituting the latter estimate in \eqref{A5} completes the proof.
\end{proof}

\subsection{Coercivity and continuity estimates}  \label{s_cont}
We  derive  coercivity and continuity estimates for the trilinear form $a_n(\cdot\, ;\cdot,\cdot)$.\begin{lemma} For arbitrary $\bz \in L^2(\Gamma_h^n)$ and $\bv_h \in \bU_h^n$ it holds
	\begin{equation}\label{lower}
		a_n(\bz;\bv_h,\bv_h)\ge \enormu{\bv_h}^2 - \xi_h \|\bP_h\bv_h\|^2_{\G{n}{h}}   \quad \text{with}~\xi_h := 1+\max_{n=0,..,N} \Vert w_N^{e,n}\bH_h(t_n) \Vert_{L^{\infty}}.
	\end{equation}
\end{lemma}
\begin{proof} From \eqref{eq:normals} and \eqref{eq:normalimproved} it follows for sufficiently small $h$ that
	\[
	|\bv|^2=|\bn_h\cdot\bv|^2+|\bP_h\bv|^2\le (1+ch^q)\big(|\tbn_h\cdot\bv|^2+|\bP_h\bv|^2\big)\quad \forall\,\bv\in\mathbb{R}^3.
	\]
Using this,  $c_n(\bz;\bv_h,\bv_h)=0$,  $\bH_h=\bP_h \bH_h \bP_h$, which follows from the definition of $\bH_h$, cf. \eqref{discrWeingarten}, and $\tau\ge 1+ch^q$ (for $h,\tau$ satisfying \eqref{CondA} and \eqref{CondTau}), we obtain the following lower estimate for $a_n(\bz;\cdot,\cdot)$:
\[
	\begin{split}
		a_n(\bz;\bv_h,\bv_h)&= \widehat a_n(\bv_h,\bv_h)=(w_N^{e,n} \bH_h \bv_h,\bv_h)_{\G{n}{h}} + \enormu{\bv_h}^2 +\tfrac\tau2\|\tbn_h\cdot\bv\|^2_{\G{n}{h}} - \tfrac12 \|\bv_h\|^2_{\Gamma_h^n}
		\\&\ge
		\enormu{\bv_h}^2 - \big(\Vert w_N^{e,n}\bH_h(t_n) \Vert_{L^{\infty}}+\tfrac12(1+ch^q)\big)\|\bP_h\bv_h\|^2_{\G{n}{h}} \quad \forall~\bv_h \in \bU_h^n,
	\end{split}
\]
from which the result follows.
\end{proof}

\begin{lemma} \label{lemcontinuity}
 The following uniform estimate holds:
 \begin{equation} \label{contestimate}
  a_n(\bz; \bu,\bv) \lesssim \big(\|\bP_h\bz\|_{L^4(\Ghn)} +1\big) \enormu{\bu}\enormu{\bv}, \quad \bz \in L^4(\Ghn), ~\bu,\bv \in \bU_h^n.
 \end{equation}
\end{lemma}
\begin{proof}
 For the terms in the trilinear form $a(\cdot\, ;\cdot,\cdot)$ that do not depend on the first argument, cf. \eqref{e:def:an}, we apply Cauchy-Schwarz inequality:
 \begin{equation} \label{h1}
   \left| \widehat a_n(\bu,\bv)\right| \lesssim  \enormu{\bu} \enormu{\bv}.
 \end{equation}
We now consider the term $\int_{\G{n}{h}}  (\bz\cdot\nabla_{\Gamma_h} \bP_h \bu)\cdot \bv \, ds_h$. The term $\int_{\G{n}{h}} (\bz\cdot\nabla_{\Gamma_h} \bP_h  \bv)\cdot \bu   \, ds_h$ can be treated similarly. Using \eqref{intestimateV} and \eqref{Kornestimate} we obtain
\begin{equation} \label{HH3}
 \begin{split}
 \Big|\int_{\G{n}{h}}  (\bz\cdot\nabla_{\Gamma_h} \bP_h \bu)\cdot \bv &\, ds_h \Big|  \lesssim \|\nabla_{\Gamma_h} \bP_h \bu\|_{\Ghn} \|\bP_h \bz\|_{L^4(\Ghn)}\|\bP_h \bv \|_{L^4(\Ghn)} \\
 & \lesssim \|\bP_h \bz\|_{L^4(\Ghn)}
 %\|\bP_h \bu\|_{\Ghn}^\frac12
 \|\bP_h\bu\|_{H^1(\Ghn)}  \|\bP_h \bv\|_{\Ghn}^\frac12 \|\bP_h\bv\|_{H^1(\Ghn)}^\frac12 \\
%  & \lesssim \|\bP_h \bz\|_{L^4(\Ghn)}\big( \|\bP_h\bu\|_{\Ghn}+ \|E_{s,h}(\bP_h \bu)\|_{\Ghn}\big)\big(\|\bP_h\bv\|_{\Ghn}+ \|E_{s,h}(\bP_h \bv)\|_{\Ghn}\big) \\
 & \lesssim \|\bP_h \bz\|_{L^4(\Ghn)}\enormu{\bu}\enormu{\bv}.
\end{split}
\end{equation}
Combining this with the same estimate for the other trilinear term and with \eqref{h1} we obtain the result \eqref{contestimate}.
\end{proof}

\subsection{Stability estimate} \label{s_stab}
First we derive a stability result for  the discrete velocity. Using this stability result and a suitable discrete inf-sup property, we then derive a stability estimate for  the discrete pressure in Theorem~\ref{pstability}.

\begin{theorem}\label{Th1}
 A solution of \eqref{NS_FD} satisfies the following estimate:
  \begin{multline}\label{FE_stab1}
    \| \bu_h^n\|^2_{\G{n}{h}} + \sum_{k=1}^{n}\|\bP_h\bu_h^k-\bu_h^{k-1}\|^2_{\G{k}{h}} +\Delta t \sum_{k=1}^{n}\big(\enorm{\bu_h^k}^2_{U^k}+2 \rho_p\|\bn_h\cdot\nabla p_h^k\|^2_{\omega_\Gamma^k}\big) \\
  \lesssim \exp(c\,t_n) \left(\!\|\bu_h^0\|^2_{\G{0}{h}}\!+ \Delta t \enorm{\bu_h^{0}}_{U^0}^2 + \Delta t
\sum_{k=0}^n \|\blf^{e,k}\|^2_{\G{k}{h}} \right),
\end{multline}
with $c$ independent of $h$, $\Delta t$ and $n$.
\end{theorem}
\begin{proof}
  We test \eqref{discrcompact}  with $\bv_h=\bu_h^n$, $q_h=- \Delta t p^n_h$. Adding the two identities leads to
 \begin{multline}\label{aux786}
  \tfrac12(\|\bP_h\bu_h^n\|^2_{\G{n}{h}}+\|\bP_h\bu_h^n-\bu_h^{n-1}\|^2_{\G{n}{h}})+ \Delta t \, a_n(\bu_h^{n-1};\bu_h^n,\bu_h^n)+ \Delta t \rho_p\|\bn_h\cdot\nabla p_h^n\|^2_{\omega_\Gamma^n} \\  =\tfrac12\|\bu_h^{n-1}\|^2_{\G{n}{h}}+\Delta t (\blf^{e,n},\bP_h \bu_h^n)_{\G{n}{h}}.
\end{multline}
We  use the lower bound \eqref{lower} and apply \eqref{fund2simple}:
\begin{align}
\|\bP_h\bu_h^n \|^2_{\G{n}{h}}&+\|\bP_h\bu_h^n-\bu_h^{n-1}\|^2_{\G{n}{h}} +{2\Delta t}\big(\enormu{\bu_h^n}^2+\rho_p\|\bn_h\cdot\nabla p_h^n\|^2_{\omega_\Gamma^n}\big) \nonumber \\
  & \le\|\bu_h^{n-1}\|^2_{\G{n}{h}}+ 2 \xi_h \Delta t\|\bP_h\bu_h^n\|^2_{\G{n}{h}} + \Delta t\|\bP_h\bu_h^n\|^2_{\G{n}{h}} + \Delta t\|\blf^{e,n}\|^2_{\G{n}{h}}   \label{est6} \\
  & \le (1 + c\,\Delta t) \|\bu_h^{n-1}\|^2_{\G{n-1}{h}} + \tfrac12\Delta t \rho_u  \|\bn_h^{n-1} \cdot \nabla \bu_h^{n-1} \|_{\mathcal{O}_\delta(\G{n-1}{h})}^2 + C \Delta t\|\bu_h^n\|^2_{\G{n}{h}} + \Delta t\|\blf^{e,n}\|^2_{\G{n}{h}}. \nonumber
\end{align}
Due to \eqref{CondTau} we have
\begin{equation} \label{A6} \begin{split}
   \tfrac12\Delta t \enormu{\bu_h^n}^2 & \geq  \tfrac12\Delta t \, \tau \|\tbn_h \cdot \bu_h^n\|^2_{\G{n}{h}}\geq   \tfrac12\Delta t \, \tau \| \bn_h \cdot \bu_h^n\|^2_{\G{n}{h}} - c \Delta t\, \tau\, h^{2q} \|\bu_h^{n}\|^2_{\G{n}{h}} \\
 & \geq \| \bn_h \cdot \bu_h^n\|^2_{\G{n}{h}} - c \Delta t h^{2q-2} \|\bu_h^{n}\|^2_{\G{n}{h}}.
\end{split} \end{equation}
Using this and $\|\bu_h^{n}\|^2_{\G{n}{h}}= \|\bP_h\bu_h^{n}\|^2_{\G{n}{h}} + \|\bn_h \cdot \bu_h^{n}\|^2_{\G{n}{h}}$ in \eqref{est6} we obtain
\begin{multline} \label{est7}
  \|\bu_h^n \|^2_{\G{n}{h}} +\|\bP_h\bu_h^n-\bu_h^{n-1}\|^2_{\G{n}{h}} +  \tfrac32\Delta t \enormu{\bu_h^n}^2+ 2 \Delta t\rho_p\|\bn_h\cdot\nabla p_h^n\|^2_{\omega_\Gamma^n} \\
  \leq  (1 + c\,\Delta t) \|\bu_h^{n-1}\|^2_{\G{n-1}{h}} + \tfrac12\Delta t \enorm{ \bu_h^{n-1}}_{U^{n-1}}^2  + C \Delta t\|\bu_h^n\|^2_{\G{n}{h}} + \Delta t\|\blf^{e,n}\|^2_{\G{n}{h}}.
\end{multline}
We  sum up the inequalities for $n=1,\dots,k$, and with $c^\ast = c+ C$ we get
\begin{multline*}
 \|\bu_h^k\|^2_{\G{k}{h}}+\sum_{n=1}^{k}\|\bP_h\bu_h^n-\bu_h^{n-1}\|^2_{\G{n}{h}}  +\Delta t\sum_{n=1}^{k}\big(\enormu{\bu_h^n}^2 + 2\rho_p\|\bn_h\cdot\nabla p_h^n\|^2_{\omega_\Gamma^n}\big)\\
 \le (1+ c \Delta t) \|\bu_h^{0}\|^2_{\G{0}{h}} + \tfrac12\Delta t \enorm{\bu_h^{0}}_{U^0}^2
+ \Delta t \sum_{n=0}^{k} c^\ast  \|\bu_h^n\|^2_{\G{n}{h}} +\Delta t \sum_{n=0}^{k}\|\blf^{e,n}\|^2_{\G{n}{h}} .
\end{multline*}
Finally, we shift the term  $\Delta t \, c^\ast \|\bu_h^k\|^2_{\G{k}{h}}$ from the right-hand side to the left hand side and apply, for $\Delta t \leq (2 c^\ast)^{-1}$,  a discrete Gronwall inequality.
\end{proof}

We now derive a stability bound for the discrete pressure. The analysis is based on results derived in \cite{olshanskii2021inf}. We recall some results from that paper.
We use the scaled $H^1(\oGn)$-norm given by
\begin{equation} \label{defpnorm}
  \|q\|_{1,{\oGn},h}:=h^\frac12 \|\nabla q\|_{\oGn}.
\end{equation}

Note that this norm is equivalent to the \emph{scaled}  $H^1(\Gamma(t_n))$ norm in the following sense. Assume $\Gamma(t_n) \subset \oGn$ and $q^e$ is the normal extension of $q \in H^1(\Gamma(t))$. Then $ h \|\nabla_\Gamma q\|_{\Gamma(t)} \sim \|q^e\|_{1,{\oGn},h}$ holds.

For a given $q_h \in Q_h^n$ we define a corresponding function $\widehat{\bv}_h \in \bU_h^n$ by
\begin{equation} \label{defvhat}
 \widehat{\bv}_h(\bx)= \sum_{E \in {\mathcal E}_{\rm reg}} h_E^2 \varphi_E(\bx) (\mathbf{t}_E \cdot \nabla q_h(\bx)) \mathbf{t}_E,
\end{equation}
with ${\mathcal E}_{\rm reg}$ being a subset of edges of tetrahedra in $\oGn$, $\mathbf{t}_E$ the unit tangent vector along the edge $E$ and $\varphi_E$ the quadratic nodal basis function corresponding to the edge $E$ (extended by zero to $\Omega$). Recall that the finite element function $\phi_h(t,\cdot)$ defines the surface approximation $\Gamma_h(t)$, cf. \eqref{discrGamma}.
% Main results derived in \cite{olshanskii2021inf} are the following:
The following result was shown in \cite[Theorem~5.9]{jankuhn2021error}: For a fixed $t_n$ and $\phi_h$ a piecewise \emph{linear} finite element function satisfying \eqref{phi_h}--\eqref{phi_hbb} there exist a subset ${\mathcal E}_{\rm reg}$ and $C$ independent of $q_h$,  discretization parameters and the position of $\Gamma_h$ in the mesh such that
\begin{align}
  \|q_h\|_{1,\oGn,h}^2 & \le C \big(b(q_h,\widehat{\bv}_h) + \rho_p \|\bn_h \cdot \nabla q_h\|_{\oGn}^2\big) \label{main1}
  \\
  \enormu{\widehat{\bv}_h} & \le C  \|q_h\|_{1,\oGn,h}. \label{main2}
\end{align}

\begin{assumption}\label{Ass1}
	We assume that \eqref{main1}--\eqref{main2}  hold for  a $P_k$ finite element function $\phi_h$ satisfying \eqref{phi_h}--\eqref{phi_hbb}, with  $C$   independent of $q_h$,  discretization parameters, the position of $\Gamma_h$ in the mesh and of $t_n$.
\end{assumption}

A  result similar to \eqref{main1}-\eqref{main2} is also known if $\phi_h=\phi$, i.e., if the geometry approximation $\Gamma_h \approx \Gamma$ is exact (cf. \cite{olshanskii2021inf}). This is a further support of the plausibility of Assumption~\ref{Ass1}. We claim that the analysis of  \cite{olshanskii2021inf} can be  extended to the case of an arbitrary polynomial degree approximation $\phi_h$ of $\phi$,  with uniform constants as specified in Assumpton~\ref{Ass1}.  Such an extension, however, will not be straighforward and is expected to be rather technical and therefore is not addressed here.

\begin{theorem} \label{pstability}
 Let Assumption~\ref{Ass1} be satisfied. Denote the upper bound in \eqref{FE_stab1} by $\bF_n=\bF_n(\bu_h^0, \mathbf{f}^e)$. The discrete problem \eqref{discrcompact} has a unique solution and for the discrete pressure solution the following holds:
 \begin{equation} \label{pestimate}
 \Delta t \sum_{k=1}^n \|p_h^k\|_{1,\omega_\Gamma^k,h} \lesssim \bF_n + \big( \frac{h}{\sqrt{\Delta t}} +1 \big) \bF_n^\frac12 \lesssim \bF_n +\bF_n^\frac12 .
 %+  {\Delta t} \sum_{k=1}^n (\bF_k+ \bF_k^\frac12).
 \end{equation}
\end{theorem}
\begin{proof}
From the discrete inf-sup property \eqref{main1} and standard results for saddle point problems  it follows that the discrete problem \eqref{discrcompact} has a unique solution. Let $p_h^n$, $n=1, \ldots,N$, be the pressure solution and, for given $n$, $\widehat{\bv}_h \in \bU_h^n$ as in \eqref{defvhat} a corresponding discrete velocity. We then have
\begin{equation} \label{h3}
 \|p_h^n\|_{1,\oGn,h}^2  \lesssim b(p_h^n,\widehat{\bv}_h) + \rho_p \|\bn_h \cdot \nabla p_h^n\|_{\oGn}^2 \lesssim  b(p_h^n,\widehat{\bv}_h) +\rho_p^\frac12 \|\bn_h \cdot \nabla p_h^n\|_{\oGn}\|p_h^n\|_{1,\oGn,h}.
\end{equation}
We use the test function $\widehat{\bv}_h$ in \eqref{discrcompact} and thus get
\[
 b(p_h^n,\widehat{\bv}_h)= - \frac{1}{\Delta t}\big((\bu_h^n, \bP_h \widehat{\bv}_h)_{\G{n}{h}} -( \bu_h^{n-1} , \bP_h \widehat{\bv}_h)_{\G{n}{h}}\big)   - a_n(\bu_h^{n-1};\bu_h^n, \widehat{\bv}_h) +(\blf^{e,n},\bP_h \widehat{\bv}_h)_{\G{n}{h}}.
\]
Now note the following:
\begin{equation} \label{HH7}
 \|\widehat{\bv}_h\|_{\Gamma_h^n}^2 \lesssim h^{-1}\|\widehat{\bv}_h\|_{\oGn}^2 \lesssim h^3 \|\nabla p_h^n\|_{\oGn}^2 \lesssim h^2 \|p_h^n\|_{1,\oGn,h}^2.
\end{equation}
Using this and the continuity estimate in \eqref{contestimate}, \eqref{A4} and \eqref{main2} we obtain
\begin{align*}
   |a_n(\bu_h^{n-1};\bu_h^n, \widehat{\bv}_h)|  &\lesssim \big( \enorm{\bu_h^{n-1}}_{U^{n-1}} + 1 \big)  \enormu{\bu_h^n} \enormu{\widehat{\bv}_h} \\ &
 \lesssim \big( \enorm{\bu_h^{n-1}}_{U^{n-1}}^2  +\enormu{\bu_h^n}^2  + \enormu{\bu_h^n}  \big) \|p_h^n\|_{1,\oGn,h}.
\end{align*}
Using this and \eqref{HH7} we get
\[ \begin{split}
 |b(p_h^n,\widehat{\bv}_h)| &  \lesssim  \frac{h}{\Delta t} \|\bP_h\bu_h^n-\bu_h^{n-1}\|_{\Ghn} \|p_h^n\|_{1,\oGn,h} +   h \|\blf^{e,n} \|_{\Ghn}\|p_h^n\|_{1,\oGn,h} \\
 & \quad + \big(  \enorm{\bu_h^{n-1}}_{U^{n-1}}^2  +\enormu{\bu_h^n}^2 + \enormu{\bu_h^n} \big) \|p_h^n\|_{1,\oGn,h}.
\end{split} \]
Hence, using \eqref{h3} we obtain
\begin{equation} \label{AA8}\begin{split}
  \|p_h^n\|_{1,\oGn,h} & \lesssim  \frac{h}{\Delta t}\|\bP_h\bu_h^n-\bu_h^{n-1}\|_{\Ghn} +  h \|\blf^{e,n} \|_{\Ghn} + \rho_p^\frac12 \|\bn_h \cdot \nabla p_h^n\|_{\oGn}\\ & \quad  + \enorm{\bu_h^{n-1}}_{U^{n-1}}^2 +\enormu{\bu_h^n}^2 + \enormu{\bu_h^n}.
\end{split} \end{equation}
Summing over $n=1,\dots,k$, yields
\begin{align}
 \Delta t \sum_{n=1}^k \|p_h^k\|_{1,\omega_\Gamma^k,h} & \lesssim {h} \sum_{n=1}^k \|\bP_h\bu_h^n-\bu_h^{n-1}\|_{\Ghn} +  h \Delta t \sum_{n=1}^k \|\blf^{e,n} \|_{\Ghn}   \label{AAAA8} \\
  & + \Delta t \sum_{n=1}^k \rho_p^\frac12 \|\bn_h \cdot \nabla p_h^n\|_{\oGn} \quad  + \Delta t \sum_{n=0}^k\enorm{\bu_h^{n}}_{U^{n}}^2  + \Delta t \sum_{n=0}^k\enorm{\bu_h^{n}}_{U^{n}} . \nonumber
\end{align}
We use the estimate \eqref{FE_stab1}, with upper bound denoted by $\bF_n$, and $k \lesssim 1/\Delta t$.  The different terms on the right-hand side can be bounded as follows, using Cauchy-Schwarz,
\begin{align}
  & h \sum_{n=1}^k \|\bP_h\bu_h^n-\bu_h^{n-1}\|_{\Ghn} \lesssim \frac{h}{\sqrt{\Delta t}}  \big(\sum_{n=1}^k \|\bP_h\bu_h^n-\bu_h^{n-1}\|_{\Ghn}^2 \big)^\frac12 \lesssim\frac{h}{\sqrt{\Delta t}} \bF_k^\frac12,  \label{PP}
\\ &  h \Delta t \sum_{n=1}^k \|\blf^{e,n} \|_{\Ghn}  + \Delta t \sum_{n=1}^k \rho_p^\frac12 \|\bn_h \cdot \nabla p_h^n\|_{\oGn}   \nonumber \\ & \lesssim h  \big(\Delta t \sum_{n=1}^k \|\blf^{e,n} \|_{\Ghn}^2 \big)^\frac12 +
 \big(\Delta  t\sum_{n=1}^k \rho_p \|\bn_h \cdot \nabla p_h^n\|_{\oGn}^2 \big)^\frac12
 \lesssim \bF_k^\frac12,  \nonumber \\
%  & \Delta t \sum_{n=0}^k  \|\bu_h^{n}\|_{\Gamma_h^{n-1}}^2 + \Delta t \sum_{n=0}^k  \|\bu_h^{n}\|_{\Gamma_h^{n-1}}
%\lesssim \Delta  t\sum_{n=1}^k (\bF_n +\bF_n^\frac12),  \\
& \Delta t \sum_{n=0}^k\enorm{\bu_h^{n}}_{U^{n}}^2 + \Delta t \sum_{n=0}^k\enorm{\bu_h^{n}}_{U^{n}} \lesssim \bF_k+ \bF_k^\frac12. \nonumber
\end{align}
Combining these estimates completes the proof.
\end{proof}

Let us discuss the stability bound in \eqref{pestimate}. Note that the bound yields stability in a discrete $L^1(H^1)$ type norm with a uniform constant for $h^2 \lesssim \Delta t$, which covers parameter choices typically used in practice for $m=1$ (the lowest order Taylor--Hood pair), namely $\Delta t \sim h^2$ (for BDF1) and $\Delta t \sim h$ (BDF2). The stability bound can be compared with bounds derived in the recent papers  \cite{burman2022eulerian,Lehrenfeldetal2021}  in which a similar discretization method for the Stokes problem on moving domains is treated. The analyses in these papers yield a uniform (in $\Delta t$ and $h$) pressure bound only for $\Delta t ^2 \sum_{k=1}^n \|p_h^k\|_{\Omega_k}$. Note the square in the scaling factor $\Delta t^2$ in this quantity. In particular these analyses do not yield
uniform bounds for $\Delta t  \sum_{k=1}^n \|p_h^k\|_{\Omega_k}$ if $\Delta t \sim h$ or $\Delta t \sim h^2$.

\begin{remark} \label{remArgument}
 \rm
 We briefly explain  why our analysis yields
 a uniform bound if $h^2 \lesssim \Delta t$.
  First note that \eqref{main1}--\eqref{main2} implies the discrete inf-sup property. In the stability analysis of the discrete pressure we do not use this discrete inf-sup property. We rather use the specific choice of the function $\widehat{\bv}_h$ in \eqref{main1}--\eqref{main2}. For this function we have, cf. \eqref{HH7},  $\|\widehat{\bv}_h\|_{\Gamma_h^n} \lesssim h \|p_h^n\|_{1,\oGn,h}$. Here we gain a power of $h$ compared to the more naive estimate $\|\widehat{\bv}_h\|_{\Gamma_h^n} \lesssim \enorm{\widehat{\bv}_h}_{U^{n}} \lesssim \|p_h^n\|_{1,\oGn,h}$. Due to this we get the factor $h$ in front of the time difference term $\|\bP_h\bu_h^n-\bu_h^{n-1}\|_{\Ghn}$ in \eqref{AA8}-\eqref{AAAA8},
  which leads to $\frac{h}{\sqrt{\Delta t}}$, instead of $\frac{1}{\sqrt{\Delta t}}$ in \eqref{PP}.  Hence we have a uniform bound for $h^2 \lesssim \Delta t$.
  It looks plausible that this approach can be used to improve the suboptimal pressure stability estimate in  \cite{Lehrenfeldetal2021} as well.
\end{remark}

\begin{remark} \label{remf}
	\rm
The treatment of non-homogeneous divergence condition  in \eqref{NSalt}  involves handling the $(f^{e,n},p_h)_{\Gamma_h^n}$ term on the right-hand side of eq. \eqref{aux786}. Standard arguments would result in an additional term
on the right hand side of the stability bound, which involves the quantity
$\|f\|_\ast=\max\limits_{n}\sup_{q_h\in Q_h^n}\frac{(f^{e,n},q_h)_{\Gamma_h^n}}{\|q_h\|_{1,\oGn,h}}$, which is dual to our discrete pressure norm.  We cannot show that $\|f\|_\ast$ remains uniformly bounded for a smooth $f$ if $h\to0$. For this reason, our stability and convergence analysis is limited  to the case of a homogeneous  divergence condition. A formal way to reduce the original problem to the one with an homogeneous divergence condition would be to solve for the potential $\psi$ satisfying $\Delta_\Gamma\psi=f$ on $\Gamma(t)$ and  substitute $\bu_T\to \bu_T+\psi$. In practice, we do not use this approach and instead apply the FEM \eqref{e:FEM}--\eqref{e:FEMa}. We see optimal order convergence results, cf. Section~\ref{s:Numerics}. 
\end{remark}

We now proceed with an error estimate. Its proof combines the arguments we used for the stability analysis  with geometric  and interpolation
error estimates. The geometric  and interpolation error estimates are treated at each time instance $t_n$ for `stationary' surfaces $\G{n}{h}$ and so results already available in the literature (cf.~\cite{reusken2015analysis,olshanskii2016trace,jankuhn2021error}) can be used. We start with a consistency estimate for \eqref{NS_FD}.

\subsection{Consistency analysis} \label{s:consistency}
In this section we derive estimates for the consistency error. The analysis  is based on a standard technique and uses estimates already available in the literature for surface vector-Laplace and surface Stokes equations.
 In the consistency and error bounds we  need estimates on derivatives of the extended solution $\bu^e(t,\bx)=\bu(t,\bp(\bx))$ in the strip $\mathcal{O}(\Gs)$. To simplify the notation, the extension of (scalar or vector-valued) functions $v$ defined on $\Gs$ is also denoted by $v$.  By differentiating the identity $v(t,\bx)=v(t,\bp(\bx))$, $(t,\bx)\in\mathcal{O}(\Gs)$, $k\ge 0$ times one finds  that  for $C^{k+1}$-smooth  manifold $\Gs$ and $v \in C^k(\Gs)$ the following bound holds:
\begin{equation}\label{u_bound_a}
\|v\|_{W^{k,\infty}(\mathcal{O}(\Gs))}\lesssim \|v\|_{W^{k,\infty}(\Gs)}.
\end{equation}
Further calculations, see, for example,  \cite[Lemma 3.1]{reusken2015analysis}, yield
\begin{equation}\label{u_bound_b}
 \|v\|_{H^{k}(U_\ep(\Gamma(t)))}\lesssim \ep^{\frac12}\|v\|_{H^{k}(\Gamma(t))}, \quad\text{for}~t\in[0,T]
\end{equation}
and any $\ep>0$ such that $U_\ep(\Gamma(t))\subset \mathcal{O}(\Gamma(t))$, where $U_\ep(\Gamma(t))$ is an $\ep$-neighborhood in $\mathbb{R}^3$.

We next observe that the smooth solution $\bu^n=\bu(t_n)$, $\bu^{n-1}=\bu(t_{n-1})=\bu(t_{n-1})^e$, $p^n=p(t_n)$ of \eqref{NSalt} satisfies the identity
\begin{equation}\label{e:consist}
\int_{\G{n}{h}}\left(\frac{\bu^n-\bu^{n-1}}{\Delta t} \right) \cdot \bP_h \bv_h\,ds+ a_n(\bu^{n-1};\bu^n,\bv_h)+ b_n(p^n,\bv_h) = \consist(\bv_h) +\int_{\Gamma^n} \mathbf{f}^n \cdot \bv_h^\ell \, ds,
\end{equation}
for any $\bv_h\in \bU^n_h$ and
with $a_n(\cdot;\cdot,\cdot)$,  $b_n(\cdot,\cdot)$ as in \eqref{e:def:an} and $\consist(\bv_h)$ collecting consistency terms due to geometric errors, time derivative approximation and linearization, i.e.
\begin{equation*}\small
\begin{split}
& \consist(\bv_h):=
\underset{I_1}{\underbrace{\int_{\G{n}{h}}\left(\frac{\bu^n-\bu^{n-1}}{\Delta t}\right)\cdot \bP_h \bv_h\, ds_h-\int_{\G{n}{}}\bu_t(t_n) \cdot \bv_h^\ell\, ds}}
+
\underset{I_2}{\underbrace{\rho_u\int_{\mathcal{O}_\delta(\G{n}{h})}((\bn_h-\bn)\cdot\nabla \bu^n)(\bn_h\cdot\nabla \bv_h) d\bx}}\\
&\quad
+\underset{I_{3}}{\underbrace{
    \tfrac12 \int_{\G{n}{h}} (\bu^{n-1}\cdot\nabla_{\Gamma_h} \bP_h \bu^n) \cdot \bv_h - (\bu^{n-1}\cdot\nabla_{\Gamma_h} \bP_h  \bv_h) \cdot  \bu^n\, ds_h
    - \tfrac12 \int_{\G{n}{}}  (\bu^n \cdot \nabla_\Gamma  \bu^n) \cdot \bv_h^\ell -  (\bu^n \cdot \nabla_\Gamma \bP\bv_h^\ell) \cdot   \bu^n \, ds}}\\
&\quad
+\underset{I_4}{\underbrace{2\mu\int_{\G{n}{h}}E_{s,h}(\bP_h\bu^n): E_{s,h} (\bP_h\bv_h)\, ds_h - 2\mu\int_{\G{n}{}}E_{s}(\bu^n): E_{s} (\bP \bv_h^\ell)\, ds}}\\
&\quad
+\underset{I_{5}}{\underbrace{\int_{\G{n}{h}} w_N^{e,n} \bv_h \cdot \bH_h \bu^n\, ds_h-\int_{\G{n}{}}\  w_N^n\bv_h^\ell \cdot \bH\bu^n\, ds}}
+
\underset{I_6}{\underbrace{\tau\int_{\G{n}{h}}(\tbn_h\cdot\bu^n)(\tbn_h\cdot\bv_h) \, ds_h-
\tau\int_{\G{n}{}}(\bn\cdot\bu^n)(\bn\cdot\bv_h^\ell)\, ds}}\\
&\quad+
\underset{I_7}{\underbrace{\int_{\G{n}{h}} \nabla_{\Gamma_h}  p^n \cdot \bv_h \,ds_h-
		\int_{\G{n}{}}\nabla_{\Gamma} p^n  \cdot \bv_h^\ell \, ds}} .
 %+ \int_{\G{n}{}}\mathbf{f}^{n} \cdot \bv_h^\ell\,  \, ds.
\end{split}
\end{equation*}

All terms above except $I_3$ have been considered in consistency analyses of TraceFEM in the literature,
\cite{lehrenfeld2018stabilized,Jankuhn2020,jankuhn2021error}. In the next lemma  we collect results which are essentially known and then treat the term $I_3$ in Lemma~\ref{l_consist2}.
\begin{lemma}\label{l_consist}  The following uniform estimates  hold
\begin{align}
|I_1| & \lesssim (\Delta t + h^q) %\|\bu\|_{W^{2,\infty}(\Gs)}
\|\bv_h\|_{\Ghn} \label{RESI1}\\
|I_2| & \lesssim  h^q %\|\bu\|_{W^{1,\infty}(\Gs)}
\rho_u^\frac12 \|\bn_h\cdot\nabla \bv_h\|_{\mathcal{O}_\delta(\G{n}{h})} \label{RESI2} \\
| I_4|&\lesssim h^{q}  %\norm{\bu}_{W^{1,\infty}(\Gs)}
\left( \|\bv_h\|_{H^1(\G{n}{h})}
+h^{-1}\|\tbn_h\cdot\bv_h\|_{\G{n}{h}}\right) \label{RESI4} \\
|I_5| & \lesssim h^{q} %\norm{\bu}_{L^\infty(\Gs)}
 \|\bv_h\|_{H^1(\G{n}{h})} \label{RESI5} \\
  |I_6|&\lesssim h^{q+1} \tau %\norm{\bu}_{L^{\infty}(\Gs)}
  \|\tbn_h\cdot\bv_h\|_{\G{n}{h}}  \lesssim  h^{q}  %\norm{\bu}_{L^{\infty}(\Gs)}
  \tau^\frac12 \|\tbn_h\cdot\bv_h\|_{\G{n}{h}}\label{RESI6}
  \\
|I_7|&\lesssim h^q %\norm{p}_{W^{1,\infty}(\Gs)}
\|\bv_h\|_{\G{n}{h}} \label{RESI7}
%|I_8|& \lesssim h^{q+1} \| \mathbf{f}^n\|_{\G{n}{h}}\|\bv_h\|_{\G{n}{h}} \label{RESI8}.
\end{align}
where $q+1$ is the order of geometry recovery defined in \eqref{phi_h} and \eqref{discrWeingarten}.
\end{lemma}
%{\bf AR: I am working on improvement of \eqref{RESI5}, using the new assumption \eqref{discrWeingarten}. The idea is to move the gradients in  $\nabla \bn$ in $I_5$ to a gradients  of  $\bv_h$ and $\bu$. Technical details still have to be worked out. In the partial integration that we need on $\Gamma_h$ jump terms appear. If this works, probably the same approach can be used to improve in Lemma~\ref{l_consist2} from $h^{q-1}$ to $h^q$}\\
\begin{proof}
 Componentwise application of the arguments in the proof of Lemma~11 in \cite{lehrenfeld2018stabilized} and $\|\bP_h-\bP\|_{L^\infty(\Gamma_h^n)} \lesssim h^q$ yields the bound in \eqref{RESI1}. In the same proof the result \eqref{RESI2} is derived,  using the assumption \eqref{CondD}. Recall that $E_{s,h} (\bP_h\bv_h)$ represents $  E_{s,h} (\bv_h) - (\bn_h\cdot \bv_h) \bH_h$, cf. Remark~\ref{remGrad}. The result \eqref{RESI4} is shown in Lemma 5.15 and (5.42) in \cite{Jankuhn2020}.
  For the estimate \eqref{RESI5} we use assumption \eqref{discrWeingarten} and an appropriate partial integration to shift the derivatives in $\bH_h$ and $\bH$ to the function $\bv_h$. Details of a proof are given in Appendix~\ref{AppendixC}. The result \eqref{RESI6} also follows with similar arguments, using \eqref{eq:normalimproved}, cf. proof of Lemma 5.18 in \cite{Jankuhn2020}. The result \eqref{RESI7} is  obtained with the same techniques. In the literature cited, bounds \eqref{RESI4}--\eqref{RESI7} were proved for the case of a stationary  surface. However,  the arguments need only norm equivalences as in \eqref{eqv1} and geometry approximation inequalities formulated in section~\ref{s:geom}. Since in our setting these  results hold uniformly in time, this implies uniform boundedness of the constants in \eqref{RESI4}--\eqref{RESI7} also with respect to time.
  %as well as to \eqref{RESI1}--\eqref{RESI2}).
\end{proof}

The nonlinear term requires a more careful handling, which is carried out below.
\begin{lemma}\label{l_consist2} It holds
%	\begin{equation}\label{est:consist2}
%			|I_3|\lesssim (\Delta t+h^{q}) \norm{\bu}_{W^{1,\infty}(\Gs)}^2\big( \|\bv_h\|_{H^1(\G{n}{h})}+ h^{-1}\|\bv_h\|_{\G{n}{h}}\big)
$
			|I_3|\lesssim (\Delta t+h^{q})\|\bv_h\|_{H^1(\G{n}{h})}.
$
%	\end{equation}
\end{lemma}
\begin{proof} Recall that the spatial-normal extension of $\bu$ to $\mathcal{O}(\Gs)$ is denoted by $\bu$, too. %Smoothness assumptions for $\Gs$ and $\bu$ give   $\bu\in W^{2,\infty}(\mathcal{O}(\Gs))$ and $\|\bu\|_{W^{2,\infty}(\mathcal{O}(\Gs))}\lesssim \|\bu\|_{W^{2,\infty}(\Gs) }$ {\bf AR: skip this?}.
We start with estimating the differences between the quadratic nonlinear terms, i.e. the third and fourth terms in $I_3$, and the corresponding linearized ones on the \emph{exact} surface at time $t_n$. For the fourth term in $I_3$ we get, using the smoothness of $\bu$
\[
\begin{split}
& \left|\int_{\G{n}{}} (\bu^n \cdot \nabla_\Gamma \bP\bv_h^\ell) \cdot \bu^n \, ds -  \int_{\G{n}{}}\left(\bu^{n-1} \cdot \nabla_\Gamma \bP\bv_h^\ell\right) \cdot  \bu^n \, ds\right| = \left|
\int_{\G{n}{}} \big( \int_{t_{n-1}}^{t_n}\bu_t dt \cdot \nabla_\Gamma \bP\bv_h^\ell \big) \cdot \bu^n \, ds \right| \\
& \qquad \qquad
\le \Delta t\,
%\|\bu_t\|_{L^\infty(\mathcal{O}(\Gs))}\|\bu\|_{L^\infty(\mathcal{O}(\Gs))}
\|\nabla_\Gamma \bP\bv_h^\ell\|_{L^2(\Gamma^n)}
\lesssim
\Delta t\,
%\|\bu\|_{W^{1,\infty}(\mathcal{O}(\Gs))}^2
\|\bv_h^\ell\|_{H^1(\Gamma^n)}
 \lesssim
\Delta t\,
%\|\bu\|_{W^{1,\infty}(\Gs)}^2
\|\bv_h\|_{H^1(\Gamma^n_h)}.
\end{split}
\]
With very similar arguments we obtain for the third term in $I_3$:
\[ \left|\int_{\G{n}{}}  (\bu^n \cdot \nabla_\Gamma  \bu^n ) \cdot \bv_h^\ell \, ds- \int_{\G{n}{}}  (\bu^{n-1} \cdot \nabla_\Gamma  \bu^n)\cdot \bv_h^\ell\, ds \right| \lesssim \Delta t
%\|\bu\|_{W^{1,\infty}(\Gs)}^2
\|\bv_h\|_{\Gamma^n_h}.
\]
For the approximate surface $\nabla_{\Gamma_h} \bP_h \bw$ represents $\nabla_{\Gamma_h} \bw - (\bn_h \cdot \bw)\bH_h$. We now compare the linearized terms on the \emph{approximate} surface $\Gamma_h^n$ (first and second one in $I_3$) with the corresponding ones on the exact surface. For the second term in $I_3$ we get
\[
\begin{split}
 & \left| \int_{\G{n}{h}} (\bu^{n-1}\cdot\nabla_{\Gamma_h} \bP_h  \bv_h)\cdot \bu^n\, ds_h - \int_{\G{n}{}}(\bu^{n-1} \cdot \nabla_\Gamma \bP\bv_h^\ell)\cdot  \bu^n \, ds \right| \\
 & \leq \underset{J_1}{\underbrace{\left| \int_{\G{n}{h}} (\bu^{n-1}\cdot\nabla_{\Gamma_h}  \bv_h)\cdot \bu^n\, ds_h - \int_{\G{n}{}}(\bu^{n-1} \cdot \nabla_\Gamma \bv_h^\ell)\cdot  \bu^n \, ds \right|}}   \\ &\quad +\underset{J_2}{\underbrace{\left|\int_{\G{n}{h}} (\bn_h \cdot \bv_h) \bu^{n-1}\cdot \bH_h \bu^n\, ds_h - \int_{\G{n}{}} (\bn \cdot \bv_h^\ell )\bu^{n-1}  \cdot \bH  \bu^n \, ds \right|}}
 \end{split}
 \]
 For the term $J_1$ we use $\nabla_{\Gamma_h}  \bv_h= \bP_h \nabla \bv_h^\ell \bP_h$, the transformation relation  \eqref{transform} applied to $\bv_h$, $\|\bP-\bP_h\|_{L^\infty(\Gamma_h^n)} \lesssim h^q$ and the bound in \eqref{eq:aux1255} for the change in surface measure. Thus  we get
\[
 J_1  \lesssim  h^{q}  \|\bv_h^\ell\|_{H^1(\Gamma^n)} \lesssim h^{q}  \|\bv_h\|_{H^1(\Gamma_h^n)}.
\]
For the term $J_2$ we proceed as in the derivation of the bound for $I_5$ in \eqref{RESI5}, cf. Appendix~\ref{AppendixC}. As in \eqref{eq45} we obtain
\begin{equation} \label{90}
 \int_{\G{n}{}} (\bn \cdot \bv_h^\ell )\bu^{n-1}  \cdot \bH  \bu^n \, ds = -\int_{\G{n}{}} (\bn\cdot \bv_h^\ell) \bn \cdot \nabla (\bP \bu^{n-1}) \bu^n \, ds.
\end{equation}
In the other term in $J_2$ we replace $\bn_h$ by $\bn$, with error bounded by $C h^q\|\bv_h\|_{\Gamma_h^n}$. With the same arguments as in \eqref{eq46}--\eqref{eq48} the resulting term $\int_{\G{n}{h}} (\bn \cdot \bv_h) \bu^{n-1}\cdot \bH_h \bu^n\, ds_h$ can be replaced by $-\int_{\Gamma^n} (\bn \cdot \bv_h^\ell) \normalbar^\ell\cdot \nabla (\bP \bu^{n-1}) \bu^n\, ds$ with error bounded by $C h^q \|\bv_h\|_{\Gamma_h^n}$. Comparing the latter term with the one on the right-hand side in \eqref{90} we get, using the assumption \eqref{discrWeingarten}, a bound $C h^q \|\bv_h\|_{\Gamma_h^n}$. Summarizing we get $J_1+J_2 \lesssim  h^{q}  \|\bv_h\|_{H^1(\Gamma_h^n)}$.

For the remaining first term in $I_3$ we use similar arguments. First note
\[
\begin{split}
 & \left| \int_{\G{n}{h}} (\bu^{n-1}\cdot\nabla_{\Gamma_h} \bP_h  \bu^n)\cdot \bv_h\, ds_h - \int_{\G{n}{}}(\bu^{n-1} \cdot \nabla_\Gamma \bu^n)\cdot \bv_h^\ell  \, ds \right| \\
 & \leq \left| \int_{\G{n}{h}} (\bu^{n-1}\cdot\nabla_{\Gamma_h}  \bu^n)\cdot \bv_h\, ds_h - \int_{\G{n}{}}(\bu^{n-1} \cdot \nabla_\Gamma \bu^n)\cdot \bv_h^\ell  \, ds \right|   +\left|\int_{\G{n}{h}} (\bn_h \cdot \bu^n) \bu^{n-1}\cdot \bH_h \bv_h\, ds_h \right| .
 \end{split}
\]
The first term $|\cdot|$ can be bounded by $Ch^q \|\bv_h\|_{\Gamma_h^n}$ using the same arguments as for $J_1$ above. For the second $|\cdot|$ term we also obtain such a bound using $\bn_h \cdot \bu^n= (\bn_h -\bn)\cdot \bu^n$ and $\|\bn_h -\bn\|_{L^\infty(\Gamma_h^n)} \lesssim h^q$.
Combining these estimates we obtain the result of the lemma.
\end{proof}

From the estimates in Lemmas~\ref{l_consist},\ref{l_consist2} and eq.~\eqref{Kornestimate} we obtain the following corollary.
\begin{corollary} \label{Cor411}
 For the consistency error the following uniform estimate holds
 \begin{equation} \label{boundEc}
 %\begin{split}
  |\consist(\bv_h) |   \lesssim %\|\bu\|_{W^{2,\infty}(\Gs)}\big(
  (\Delta t + h^q)\enormu{\bv_h}
  %\big) \\ & \quad + \|p\|_{W^{1,\infty}(\Gs)} h^q \|\bv_h\|_{\Ghn}
  \quad \text{for all}~\bv_h \in \bU_h^n.
 %\end{split}
\end{equation}
\end{corollary}

\subsection{Error estimates} \label{sec:aprioriest}
In  this section we derive a bound for the velocity error  $\err^n:=\bu^n-\bu^n_h \in H^1(\mathcal{O}_\delta(\G{n}{h}))$ and for the pressure error   $r^n:=p^n-p^n_h$, $r^n\in H^1(\omega_{\Gamma}^n)$. From \eqref{discrcompact} and \eqref{e:consist} we get the error equation, for arbitrary $\bv_h\in \bU^n_h$,
\begin{equation}\label{e:err}
\tfrac{1}{\Delta t}(\err^n-\err^{n-1}, \bP_h \bv_h)_{\Ghn}
+ c_n(\err^{n-1}; \bu^n,\bv_h) + a_n(\bu_h^{n-1}; \err^n,\bv_h) + b_n(r^n,\bv_h) = \consist(\bv_h)+ \delta_f^n(\bv_h),
\end{equation}
with the data error term $\delta_f^n(\bv_h):= \int_{\Gamma^n} \mathbf{f}^n \cdot \bv_h^\ell \, ds -\int_{\Ghn} \mathbf{f}^{e,n} \bP_h \cdot \bv_h^\ell \, ds_h$.  Let $\bu_I^n\in \bU_h^n$ and $p_I^n \in Q_h^n$ be the
nodal interpolants  for $\bu^n$ in $\mathcal{O}_\delta(\G{n}{h})$ and $p^n$ in $\omega_\Gamma^n$, respectively.

We proceed using standard techniques, based on  splitting the error $\err^n$ into approximation and finite element parts,
\[
\err^n=\underset{\mbox{$\be_I^n$}}{\underbrace{(\bu^n-\bu^n_I)}}+\underset{\mbox{$\be^n_h$}}{\underbrace{(\bu^n_I-\bu^n_h)}}.
\]
Using
$a_n(\bu_h^{n-1}; \err^n,\bv_h)= \widehat a_n(\be_I^n,\bv_h)+\widehat a_n(\be_h^n,\bv_h) + c_n(\bu_h^{n-1}; \err^n,\bv_h) $
equation \eqref{e:err} can be reformulated as
\begin{equation}\label{e:err1}
\tfrac{1}{\Delta t} (\be^n_h-\be_h^{n-1}, \bP_h \bv_h)_{\Ghn}
+ \widehat a_n(\be^n_h,\bv_h) + b_n(p^n_I-p^n_h,\bv_h) =\consist(\bv_h)+ \interpol(\bv_h)+\mathcal{C}^n(\bv_h)+\delta_f^n(\bv_h),
\end{equation}
with the interpolation and nonlinear terms
\[
\begin{split}
\interpol(\bv_h)&:=-\tfrac{1}{\Delta t} (\be_I^n-\be_I^{n-1},\bP_h \bv_h)_{\Ghn}- \widehat a_n(\be_I^n,\bv_h)- b_n(p^n-p^n_I,\bv_h),\\
\mathcal{C}^n(\bv_h)&:= -c_n(\err^{n-1};\bu^n,\bv_h)-c_n(\bu^{n-1}_h;\err^n,\bv_h).
\end{split}
\]
An  estimate for the interpolation terms is given in the following lemma.
\begin{lemma}\label{l_interp} It holds
\begin{equation}\label{est_inter}
  |\interpol(\bv_h)|\lesssim (h^{m+1}+h^q) \, %\left(\norm{\bu}_{W^{m+2,\infty}(\Gs)}+\norm{p}_{W^{m+1,\infty}(\Gs)}\right) \,
  \enorm{\bv_h}_{U^n}, \quad \bv_h \in \bU_h^n.
\end{equation}
\end{lemma}
\begin{proof} For this to hold it is important to use the nodal interpolation. This  operator has the key property that it interpolates velocity in the fixed (i.e., time independent) finite element space $\bU_h$,  which allows to shift time differentiation from outside the  interpolation  operator to inside. Using this approach, which is standard in the analysis of parabolic problems on Euclidean domains,  one  avoids a factor $\tfrac{1}{\Delta t}$ while handling the first term in $\interpol(\bv_h)$, cf. \cite[Lemma~12]{lehrenfeld2018stabilized}.
Estimates of the remaining two terms in $\interpol(\bv_h)$ follow by standard arguments of applying the Cauchy-Schwarz and uniform trace FE interpolation inequalities~\cite{olshanskii2016trace}; see e.g. \cite[\S~5.3]{Jankuhn2020}
\end{proof}

For the nonlinear terms a bound is given in the following lemma.

\begin{lemma}\label{l_nonlinear}  The following holds, with $\bv_h \in \bU^n_h$,
	\begin{align}\label{est_interC1}
		|\mathcal{C}^n(\bv_h)| & \lesssim \big( h^{m+2} +\|\be_h^{n-1}\|_{\Gamma_h^{n-1}}+h \enorm{\be_h^{n-1}}_{U^{n-1}}+ \enorm{\bu_h^{n-1}}_{U^{n-1}}^{\frac12} (h^{m+1}+ \enorm{\be_h^{n}}_{U^{n}}  )\big) \enormu{\bv_h} \\
		\label{est_interC2}
		|\mathcal{C}^n(\be_h^n)| & \lesssim \big( h^{m+2} + \|\be_h^{n-1}\|_{\Gamma_h^{n-1}}+h \enorm{\be_h^{n-1}}_{U^{n-1}} +\enorm{\bu_h^{n-1}}_{U^{n-1}}^{\frac12} h^{m+1}  \big)\enormu{\be_h^n}.
	\end{align}
\end{lemma}
\begin{proof}
	We estimate, using \eqref{fund3},
\[
\begin{split}
|c_n(\err^{n-1};\bu^n,\bv_h)|&\lesssim (\|\be_I^{n-1}\|_{\Gamma_h^n}+ \|\be^{n-1}_h\|_{\Gamma_h^n}) \|\bu^n\|_{W^{1,\infty}(\Gamma_h^n)} \enormu{\bv_h}\\
&\lesssim  (h^{m+2}\|\bu^n\|_{W^{m+2,\infty}} + \|\be^{n-1}_h\|_{\Gamma_h^n})\|\bu^n\|_{W^{1,\infty}}\enormu{\bv_h}\\
& \lesssim  (h^{m+2} + \|\be^{n-1}_h\|_{\Gamma_h^{n-1}}+\delta_{n-1}^\frac12 \|\bn_h^{n-1} \cdot \nabla\be^{n-1}_h \|_{\mathcal{O}_\delta(\G{n-1}{h})} )\enormu{\bv_h} \\
& \lesssim  (h^{m+2} + \|\be^{n-1}_h\|_{\Gamma_h^{n-1}}+ h \enorm{\be_h^{n-1}}_{U^{n-1}} )\enormu{\bv_h}.
%&\le C(h^m+h^q)\|\be^n_h\|_{U^n} \le  C(h^{2m}+h^{2q}) + \frac18\|\be^n_h\|_{U^n}^2
%\\
%|a_n(\be^{n-1}_h;\bu^n,\be_h^n)|&\le
% \|\be^{n-1}_h\|_{\Gamma_h^n}\|\bu^n\|_{W^{1,\infty}(\Gamma_h^n)}\|\be^n_h\|_{U^n} \le
%  C \|\be^{n-1}_h\|_{\Gamma_h^n}^2+\frac18\|\be^n_h\|_{U^n}^2.
\end{split}
\]
For the other term in $\mathcal{C}^n(\bv_h)$ we use the estimates in \eqref{HH3}, \eqref{A3}, \eqref{intestimateV} and \eqref{Kornestimate}
\[
\begin{split}
 |c_n(\bu^{n-1}_h;\err^n,\bv_h)|   &\lesssim \|\bu_h^{n-1}\|_{L^4(\Gamma_h^{n})} \big(h^{m+1} +\enormu{\be_h^n} \big)\enormu{\bv_h}\\
 &\lesssim (\|\bu_h^{n-1}\|_{L^4(\Gamma_h^{n-1})}+\|\bn_h\cdot\nabla\bu_h^{n-1}\|_{\mathcal{O}_\delta(\G{n-1}{h})}) \big(h^{m+1} +\enormu{\be_h^n} \big)\enormu{\bv_h}\\
 &\lesssim \big(\|\bu_h^{n-1}\|_{\Gamma_h^{n-1}}^{\frac12}\enorm{\bu_h^{n-1}}_{U^{n-1}}^{\frac12} + h^\frac12\enorm{\bu_h^{n-1}}_{U^{n-1}}\big)\big(h^{m+1} + \enormu{\be_h^n} \big)\enormu{\bv_h} \\
 & \lesssim \big(\|\bu_h^{n-1}\|_{\Gamma_h^{n-1}}^{\frac12} + h^\frac12\enorm{\bu_h^{n-1}}_{U^{n-1}}^\frac12\big)\enorm{\bu_h^{n-1}}_{U^{n-1}}^{\frac12}\big(h^{m+1} + \enormu{\be_h^n} \big)\enormu{\bv_h}.
\end{split}
\]
From the stability result in Theorem~\ref{Th1} and $h^2\lesssim \Delta t$ we get $\|\bu_h^{n-1}\|_{\Gamma_h^{n-1}}\leq C$ and $h^\frac12 \enorm{\bu_h^{n-1}}_{U^{n-1}}^\frac12 \lesssim \Delta t^\frac14 \enorm{\bu_h^{n-1}}_{U^{n-1}}^\frac12 \leq C$,
with $C$ specified in \eqref{FE_stab1}, in particular depending only on the data.
Combining these results yields \eqref{est_interC1}.
If $\bv_h=\be_h^n$ we use skew-symmetry and obtain
\[
|c_n(\bu^{n-1}_h;\err^n,\be_h^n)|=|c_n(\bu^{n-1}_h;\be_I^n,\be_h^n)| \lesssim  \enorm{\bu_h^{n-1}}_{U^{n-1}}^{\frac12} h^{m+1} \enormu{\be_h^n}.
\]
\end{proof}

Now we are prepared to prove the main result of the paper. Let $\bu_h^0=\bu_I^0\in \bU_h^0$ be a suitable interpolant to
$\bu^0\in \mathcal{O}(\Gamma_h^0)$.
\begin{theorem}\label{Th2AR} Let Assumption~\ref{Ass1} be satisfied. Let $(\bu,p)$ be the solution of \eqref{NSalt}. Let  $\bu_h^n$, $p^n_h$, $n=1,\dots,N$,
 be the finite element solution of \eqref{discrcompact}. For the errors
   $\err^n = \bu_h^n - \bu^n$, $r^n=p^n_h-p^n$  the following estimate holds:
  \begin{align}\label{FE_est1AR}
    \|\err^n\|^2_{\G{n}{h}}+\tfrac18 {\Delta t}\sum_{k=1}^{n}\big( \enorm{\err^k}_{U^k}^2+\rho_p\|\bn_h\cdot\nabla r^k\|^2_{\omega_\Gamma^n}\big)
     & \lesssim \exp(c\,t_n) %R(u,p)
    \big(\Delta t^2+ h^{2(m+1)} + h^{2q}\big),\\
    \Delta t \sum_{k=1}^n \|r^k\|_{1,\oGn,h} & \lesssim \exp(c\,t_n) %R(u,p)
    \big(\Delta t+ h^{m+1} + h^{q}\big), \label{Ep}
  \end{align}
  with $c$ depending on the problem data and  independent of $h$, $\Delta t$, $n$ and of the positions of the surface in the background mesh.
\end{theorem}
\begin{proof}
The arguments used to prove \eqref{FE_est1AR} largely repeat those used to show the stability result in  Theorem~\ref{Th1} and involve estimates from Corollary~\ref{Cor411}, Lemmas~\ref{l_interp} and \ref{l_nonlinear} to bound the  arising right-hand side terms.
We set $\bv_h=\Delta t \be^n_h$ in \eqref{e:err1}. This yields
\begin{equation} \label{AA1}  \begin{split}
 & \|\bP_h\be_h^n\|^2_{\G{n}{h}} +\|\bP_h\be_h^n-\be_h^{n-1}\|^2_{\G{n}{h}}+{2\Delta t} \enormu{\be_h^n}^2 + {2\Delta t}\, b_n(p_I^n-p_h^n,\be_h^n)\\
& \leq \|\be_h^{n-1}\|^2_{\G{n}{h}} + c \Delta t\|\be_h^{n}\|^2_{\G{n}{h}} + 2 \Delta t \big(|\interpol(\be_h^n)|+|\consist(\be_h^n)|+|\mathcal{C}^n(\be_h^n)|+|\delta_f^n(\be_h^n)|\big).
 \end{split}
\end{equation}
As in \eqref{A6} we get
\[
  \|\be_h^n\|^2_{\G{n}{h}}\leq \|\bP_h\be_h^n\|^2_{\G{n}{h}} +\tfrac12 \Delta t \enormu{\be_h^n}^2 + c \Delta t h^{2q}  \|\be_h^n\|^2_{\G{n}{h}}.
\]
Using this, \eqref{fund2simple} and the relation
\[
b_n(p^n_I-p_h^n, \be^n_h)=b_n(p^n_I-p_h^n, \bu_I^n)+\rho_p\|\bn_h\cdot\nabla (p_h^n-p^n_I)\|_{\omega_\Gamma^n}^2 + \underbrace{\rho_p \big(\bn_h\cdot\nabla p_I^n,\bn_h\cdot\nabla (p_h^n-p^n_I)\big)_{\omega_\Gamma^n}}_{\mathcal{E}_1(p_h^n-p^n_I)}
\]
one gets
\begin{equation} \label{A7}
 \begin{split}
  & \|\be_h^n\|^2_{\G{n}{h}} +\|\bP_h\be_h^n-\be_h^{n-1}\|^2_{\G{n}{h}}+\tfrac32 \Delta t \enormu{\be_h^n}^2 + {2\Delta t}\rho_p\|\bn_h\cdot\nabla (p_h^n-p^n_I)\|_{\omega_\Gamma^n}^2 \\
& \leq \|\be_h^{n-1}\|^2_{\G{n-1}{h}} + c \Delta t \big( \|\be_h^{n-1}\|^2_{\G{n-1}{h}}+\|\be_h^{n}\|^2_{\G{n}{h}}\big)+\tfrac12\Delta t\enorm{\be_h^{n-1}}_{U^{n-1}}^2 \\
& \quad + 2 \Delta t \big(|\interpol(\be_h^n)|+|\consist(\be_h^n)|+|\mathcal{C}^n(\be_h^n)|+|\delta_f^n(\be_h^n)| + |b_n(p^n_I-p_h^n, \bu_I^n)| +|\mathcal{E}_1(p_h^n-p^n_I)|\big).
 \end{split}
\end{equation}
We estimate the different terms in the last line of \eqref{A7}. From \eqref{est_inter} we get
\begin{equation} \label{t1}
 |\interpol(\be_h^n)| \leq  C(h^{m+1} + h^q)\enormu{\be_h^n} \leq  \widetilde C( h^{2m+2} +h^{2q}) + \tfrac{1}{16}  \enormu{\be_h^n}^2.
\end{equation}
For the second term we obtain, using \eqref{boundEc}:
\begin{equation} \label{t2}
 |\consist(\be_h^n)|   \leq  C( \Delta t  + h^{q}) \enormu{\be_h^n})   \leq \widetilde C (\Delta t^2 +h^{2q}) +\tfrac{1}{16}  \enormu{\be_h^n}^2 .
\end{equation}
For the third term we have the result \eqref{est_interC2}:
\begin{equation} \label{t3} \begin{split}
 |\mathcal{C}^n(\be_h^n)| & \leq  C (h^{m+2}+  \|\be_h^{n-1}\|_{\Gamma_h^{n-1}}+ h \enorm{\be_h^{n-1}}_{U^{n-1}}+ h^{m+1}\enorm{\bu_h^{n-1}}_{U^{n-1}}^{\frac12} )\enormu{\be_h^n} \\
 &\leq   \widetilde C\big( h^{2m+2}(1+ \enorm{\bu_h^{n-1}}_{U^{n-1}}) + \|\be_h^{n-1}\|_{\Gamma_h^{n-1}}^2+ h^2 \enorm{\be_h^{n-1}}_{U^{n-1}}^2\big) + \tfrac{1}{16}\enormu{\be_h^n}^2 .
                            \end{split}
\end{equation}
For the data error term we have $|\delta_f^n(\be_h^n)| \leq C h^q \|\be_h^n\|_{\Ghn} \leq \tilde C h^{2q} +\tfrac12 \|\be_h^n\|_{\Ghn}^2$. Furthermore we have, using $\bn \cdot \nabla p^n=0$ and $\|\bn_h\cdot\nabla (p^n_I- p^n)\|_{\omega_\Gamma^n} \lesssim h^{m} \|p^n\|_{H^{m+1}(\omega_\Gamma^n)} \lesssim h^{m+\frac12}\|p^n\|_{H^{m+1}(\Gamma^n)}$:
\begin{equation} \label{t5} \begin{split}
|\mathcal{E}_1(p_h^n-p^n_I)|&=\rho_p  \left| \big(\bn_h\cdot\nabla (p^n_I- p^n),\bn_h\cdot\nabla (p_h^n-p_I^n)\big)_{\omega_\Gamma^n}
+ \big((\bn_h-\bn)\cdot\nabla p^n,\bn_h\cdot\nabla (p_h^n-p_I^n)\big)_{\omega_\Gamma^n} \right| \\
&\leq  C\rho_p (h^{2m+1}+h^{2q})
+\tfrac12 \rho_p\|\bn_h\cdot\nabla (p_h^n-p^n_I)\|^2_{\omega_\Gamma^n}\\
&\leq  C (h^{2m+2}+h^{2q+1})+\tfrac12 \rho_p \|\bn_h\cdot\nabla (p_h^n-p^n_I)\|^2_{\omega_\Gamma^n}.
\end{split} \end{equation}
It remains to estimate the term $|b_n(p^n_I-p_h^n, \bu_I^n)|$ in the last line of \eqref{A7}. We use the splitting
\[
 |b_n(p^n_I-p_h^n, \bu_I^n)| \leq |b_n(p^n_I-p_h^n, \bu_I^n-\bu)| + |b_n(p^n_I-p_h^n, \bu)|
\]
For the first term on the right-hand side we get
\begin{equation} \begin{split}
 |b_n(p^n_I-p_h^n, \bu_I^n-\bu)| & \lesssim h^{m+2}\|\nabla_{\Gamma_h}(p^n_I-p_h^n)\|_{\Ghn} \lesssim h^{m+\frac32} \|\nabla(p^n_I-p_h^n)\|_{\omega_\Gamma^n} \\ & \lesssim h^{m+1} \|p^n_I-p_h^n\|_{1,{\oGn},h}.
\end{split} \end{equation}
For the second term  we define $q_h:=p^n_I-p_h^n \in  H^1(\Gamma_h)$ and note the following. For the lifting of this function  we have $\int_\Gamma \nabla_\Gamma q_h^\ell  \cdot \bu \, ds=0$. Recall the transformation formula $\nabla_{\Gamma_h} q_h(\by)= \bP_h(\by)(\bP(\by) - d(\by) \bH(\by))\nabla q^\ell(\bp(\by))$.  Using this and $\bP \bu =\bu$ we get
\[
  b_n(q_h, \bu) =  \int_{\Gamma_h^n}   \bP\bP_h(\bP - d \bH)\nabla q^\ell(\bp(\cdot))\cdot \bu \, ds_h . \]
Using $\|d\|_{L^\infty(\Gamma_h^n)} \lesssim h^{q+1}$ and $\| \bP\bP_h\bP -\bP\|_{L^\infty(\Gamma_h^n)} = \| \bP \bn_h \bn_h^T\bP\|_{L^\infty(\Gamma_h^n)} \lesssim h^{2q}$ we get
\[
  \left| b_n(q_h, \bu) - \int_{\Gamma_h^n} \bP\nabla q_h^\ell(\bp(\cdot))\cdot \bu \, ds_h\right| \lesssim h^{q+1} \|\nabla_{\Gamma_h} q_h \|_{\Gamma_h^n}.
\]
Using $\int_\Gamma \nabla_\Gamma q_h^\ell  \cdot \bu \, ds=0$ and the bound for the change in surface measure in \eqref{eq:aux1255} we obtain $\left| \int_{\Gamma_h^n} \bP\nabla q_h^\ell(\bp(\cdot))\cdot \bu \, ds_h\right| \lesssim h^{q+1}  \|\nabla_{\Gamma_h} q_h \|_{\Gamma_h^n}$. Combining these results yields the estimate
\[
  |b_n(p^n_I-p_h^n, \bu)| \lesssim h^{q+1} \|\nabla_{\Gamma_h}(p^n_I-p_h^n)\|_{\Ghn} \lesssim h^{q+\frac12} \|\nabla(p^n_I-p_h^n)\|_{\omega_\Gamma^n} \lesssim h^{q} \|p^n_I-p_h^n\|_{1,{\oGn},h}.
\]
Collecting these results, we see that the term between brackets $(\ldots)$ in the last line in
 \eqref{A7} can be bounded by
\begin{equation} \label{AAA} \begin{split}
  & \tfrac{3}{16} \enormu{\be_h^n}^2+ \tfrac12 \rho_p \|\bn_h\cdot\nabla (p_h^n-p^n_I)\|^2_{\omega_\Gamma^n} + C\|\be_h^{n-1}\|_{\Gamma_h^{n-1}}^2+ C h^2 \enorm{\be_h^{n-1}}_{U^{n-1}}^2 \\ &
  + \tilde C\big( \Delta t^2 + h^{2q}+ h^{2m+2}(1+ \enorm{\bu_h^{n-1}}_{U^{n-1}} ) + (h^{m+1}+h^{q})\|p^n_I-p_h^n\|_{1,{\oGn},h}\big).
\end{split} \end{equation}
The first two terms in \eqref{AAA} can be shifted to the left in \eqref{A7}. We assume that $h$ is sufficiently small such that  $C h^2 \leq \tfrac{1}{16}$. Then we get
\begin{equation} \label{A8}
 \begin{split}
  & \|\be_h^n\|^2_{\G{n}{h}} +\|\bP_h\be_h^n-\be_h^{n-1}\|^2_{\G{n}{h}}+ \Delta t \enormu{\be_h^n}^2 + {\Delta t}\rho_p\|\bn_h\cdot\nabla (p_h^n-p^n_I)\|_{\omega_\Gamma^n}^2 \\
& \leq \|\be_h^{n-1}\|^2_{\G{n-1}{h}} + c \Delta t \big( \|\be_h^{n-1}\|^2_{\G{n-1}{h}}+\|\be_h^{n}\|^2_{\G{n}{h}}\big)+\tfrac58\Delta t\enorm{\be_h^{n-1}}_{U^{n-1}}^2 \\
&  + C \Delta t \Big( \Delta t^2 + h^{2q}+ h^{2m+2}(1+ \enorm{\bu_h^{n-1}}_{U^{n-1}} )  + (h^{m+1}+h^{q})\|p^n_I-p_h^n\|_{1,{\oGn},h}\Big) .
 \end{split}
\end{equation}
For deriving a bound for $\|p^n_I-p_h^n\|_{1,{\oGn},h}$ we use the same approach as in the proof of Theorem~\ref{pstability}. For the given $q_h:=p^n_I-p_h^n \in Q_h^n$ we take the corresponding $\widehat{\bv}_h \in \bU_h^n$ as in \eqref{defvhat}, for which the estimates \eqref{main1} and \eqref{main2} hold. We take $\bv_h=\widehat{\bv}_h$ in \eqref{e:err1}. For the right-hand side in \eqref{e:err1} we introduce the notation $G_h(\bv_h):= \consist(\bv_h)+ \interpol(\bv_h)+\mathcal{C}^n(\bv_h)+\delta_f^n(\bv_h)$. As in \eqref{AA8} we then get
\begin{equation} \label{AAA8} \begin{split}
 \|p_I^n -p_h^n\|_{1,\oGn,h} & \lesssim  \frac{h}{\Delta t}\|\bP_h\be_h^n-\be_h^{n-1}\|_{\Ghn} + \rho_p^\frac12 \|\bn_h \cdot \nabla (p_I^n -p_h^n)\|_{\oGn} \\
 &+ \enorm{\be_h^{n-1}}_{U^{n-1}}^2 + \enorm{\be_h^{n}}_{U^{n}}^2 +\enorm{\be_h^{n}}_{U^{n}}+ |G_h(\widehat{\bv}_h)|\|p_I^n -p_h^n\|_{1,\oGn,h}^{-1}.
\end{split} \end{equation}
We estimate the terms in $|G_h(\widehat{\bv}_h)|$:
\begin{align*}
 |\consist(\widehat{\bv}_h)| & \lesssim (\Delta t + h^q)\enormu{\widehat{\bv}_h}    \lesssim (\Delta t + h^q)  \|p^n_I-p_h^n\|_{1,{\oGn},h}  \\
  |\interpol(\widehat{\bv}_h)|& \lesssim (h^{m+1}+ h^q) \enormu{\widehat{\bv}_h}\lesssim (h^{m+1}+ h^q)\|p^n_I-p_h^n\|_{1,{\oGn},h}  \\
  |\mathcal{C}^n(\widehat{\bv}_h)| & \lesssim  \big(  h^{m+2} +\|\be_h^{n-1}
  \|_{\Gamma_h^{n-1}}+ h \enorm{\be_h^{n-1}}_{U^{n-1}} + \enorm{\bu_h^{n-1}}_{U^{n-1}}^{\frac12}(h^{m+1} +  \enorm{\be_h^{n}}_{U^{n}} ) \big)\enormu{\widehat{\bv}_h}\\
            & \lesssim  \big(  h^{m+2} + \enorm{\be_h^{n-1}}_{U^{n-1}}  + \enorm{\bu_h^{n-1}}_{U^{n-1}}^{\frac12}(h^{m+1} +  \enorm{\be_h^{n}}_{U^{n}} ) \big) \|p^n_I-p_h^n\|_{1,{\oGn},h} \\
  |\delta_f^n(\widehat{\bv}_h)| & \lesssim h^q \|\widehat{\bv}_h\|_{\Ghn} \lesssim h^{q+1}\|p^n_I-p_h^n\|_{1,{\oGn},h}.
\end{align*}
%5Using
%\[ |\widehat a_n(\be_h^n,\widehat{\bv}_h)| \lesssim \enormu{\be_h^n} \enormu{\widehat{\bv}_h} \lesssim  \enormu{\be_h^n}\|p^n_I-p_h^n\|_{1,{\oGn},h}
%\]
%and the arguments as in the proof of Theorem~\ref{pstability}  we obtain
Thus we obtain
\begin{equation} \label{AA9} \begin{split}
 \|p_I^n -p_h^n\|_{1,\oGn,h} & \lesssim  \frac{h}{\Delta t}\|\bP_h\be_h^n-\be_h^{n-1}\|_{\Ghn} + \rho_p^\frac12 \|\bn_h \cdot \nabla (p_I^n -p_h^n)\|_{\oGn} \\ &
 + \enorm{\be_h^{n-1}}_{U^{n-1}}^2 + \enorm{\be_h^{n}}_{U^{n}}^2 + \enorm{\be_h^{n-1}}_{U^{n-1}}+ \enorm{\be_h^{n}}_{U^{n}} \\ & +\Delta t + h^{m+1} +h^q+ \enorm{\bu_h^{n-1}}_{U^{n-1}}^{\frac12}(h^{m+1} +  \enorm{\be_h^{n}}_{U^{n}} ).
 \end{split} \end{equation}
 Now note that
 \[  C \Delta t (h^{m+1}+ h^q)\enorm{\bu_h^{n-1}}_{U^{n-1}}^{\frac12}(h^{m+1} +  \enorm{\be_h^{n}}_{U^{n}} )  \leq \tfrac{\Delta t}{16}  \enorm{\be_h^{n}}_{U^{n}}^2 +\widetilde C \Delta t (h^{2q}+h^{2m+2}) \left(1+ \enorm{\bu_h^{n-1}}_{U^{n-1}}\right).
 \]
Using this, \eqref{AA9} and $h^2 \lesssim \Delta t$  we obtain  (for $h$ sufficiently small) for the last term in \eqref{A8}:
 \[ \begin{split}
    & C \Delta t (h^{m+1} + h^{q}) \|p_I^n -p_h^n\|_{1,\oGn,h} \le \tfrac12 \|\bP_h\be_h^n-\be_h^{n-1}\|_{\Ghn}^2 + \tfrac12 \Delta t \rho_p \|\bn_h \cdot \nabla (p_I^n -p_h^n)\|_{\oGn}^2  \\
   & \quad +\tfrac18 \Delta t\big( \enorm{\be_h^{n-1}}_{U^{n-1}}^2 + \enorm{\be_h^{n}}_{U^{n}}^2\big) + \tilde C \Delta t\big( \Delta t^2+ (h^{2q}+h^{2m+2})(1+ \enorm{\bu_h^{n-1}}_{U^{n-1}} ) \big).
 \end{split} \]
Substituting this in \eqref{A8} and shifting terms from the right to the left-hand side we get
\begin{equation} \label{A9}
 \begin{split}
  \|\be_h^n\|^2_{\G{n}{h}} &+\tfrac12 \|\bP_h\be_h^n-\be_h^{n-1}\|_{\Ghn}^2+ \tfrac78 \Delta t \enormu{\be_h^n}^2 + \tfrac12 {\Delta t}\rho_p\|\bn_h\cdot\nabla (p_h^n-p^n_I)\|_{\omega_\Gamma^n}^2 \\
& \leq \|\be_h^{n-1}\|^2_{\G{n-1}{h}} + c_1 \Delta t \big( \|\be_h^{n-1}\|^2_{\G{n}{h}}+\|\be_h^{n}\|^2_{\G{n}{h}})+\tfrac68\Delta t\enorm{\be_h^{n-1}}_{U^{n-1}}^2\\
&\quad+ c_2 \Delta t\big( \Delta t^2+ (h^{2q}+h^{2m+2})(1+ \enorm{\bu_h^{n-1}}_{U^{n-1}}) \big) .
 \end{split}
\end{equation}
  We drop the term $\tfrac12 \|\bP_h\be_h^n-\be_h^{n-1}\|_{\Ghn}^2$, sum over $n=1,\dots,k$, apply the discrete Gronwall inequality and use the discrete stability estimate $\Delta t\sum_{n=1}^{k}\enorm{\bu_h^{n-1}}_{U^{n-1}}\lesssim \big(\Delta t\sum_{n=1}^{k}\enorm{\bu_h^{n-1}}_{U^{n-1}}^2\big)^{\frac12} \leq C$  to get
\begin{equation} \label{BB}
  Q_{e,k}:=  \|\be_h^k \|^2_{\G{k}{h}}+\tfrac18 \Delta t \sum_{n=1}^{k} \big(\enormu{\be_h^n}^2+\rho_p\|\bn_h\cdot\nabla (p_h^n-p^n_I)\|^2_{\omega_\Gamma^n}\big)  \lesssim \exp(c\, t_k)
                (\Delta t^2+h^{2m+2} + h^{2q}).
\end{equation}
The triangle inequality and standard FE interpolation properties give
\begin{equation*}
\begin{split}
    \|\err^k\|^2_{\G{k}{h}}&+  \tfrac18 \Delta t \sum_{n=1}^{k}\big(\enormu{\err^n}^2+\rho_p\|\bn_h\cdot\nabla r^n\|^2_{\omega_\Gamma^n}\big)
    \\ &
    \leq
    2 Q_{e,k} +2\|\be^k\|^2_{\G{k}{h}}+\tfrac14 \Delta t \sum_{n=1}^{k}  \big(\enormu{\be^n}^2+\rho_p\|\bn_h\cdot\nabla (p^n-p^n_I)\|^2_{\omega_\Gamma^n}\big)
    ~
    \lesssim Q_{e,k}+C h^{2m+2}.
\end{split}
\end{equation*}
  This completes the proof of \eqref{FE_est1AR}.

  We now derive a pressure error bound. The approach is similar to the one used in the proof of Theorem~\ref{pstability}. First note that if we do not drop the term $\tfrac12 \|\bP_h\be_h^n-\be_h^{n-1}\|_{\Ghn}^2$  when going from \eqref{A9} to \eqref{BB} we get
  \begin{equation}
   \tilde Q_{e,k}:= Q_{e,k}+ \tfrac12 \sum_{n=1}^k \|\bP_h\be_h^n-\be_h^{n-1}\|_{\Ghn}^2 \lesssim \exp(c\, t_k)(\Delta t^2+h^{2m+2} + h^{2q})=: F_k.
  \end{equation}
From \eqref{AA9} we obtain, with the same arguments as in the proof of Theorem~\ref{pstability}:
\[
  \Delta t \sum_{n=1}^k \|p_I^n -p_h^n\|_{1,\oGn,h} \lesssim \left(\tfrac{h}{\sqrt{\Delta t}} +1\right) \tilde Q_{e,k}^\frac12 + Q_{e,k} + \Delta t + h^{m+1} + h^q  +\Delta t \sum_{n=1}^k \enorm{\bu_h^{n-1}}_{U^{n-1}}^{\frac12}(h^{m+1} +  \enorm{\be_h^{n}}_{U^{n}})
\]
Similar to above, using discrete stability for the last term we get
\[
  \Delta t \sum_{n=1}^k \enorm{\bu_h^{n-1}}_{U^{n-1}}^{\frac12}(h^{m+1} +  \enorm{\be_h^{n}}_{U^{n}}) \lesssim h^{m+1} + Q_{e,k}^\frac12.
\]
Recall that $h^2 \lesssim \Delta t$. Hence, we obtain
\[
  \Delta t \sum_{n=1}^k \|p_I^n -p_h^n\|_{1,\oGn,h} \lesssim \tilde Q_{e,k}^\frac12 + \tilde Q_{e,k} + \Delta t + h^{m+1} + h^q \lesssim F_k^\frac12 +F_k\lesssim F_k^\frac12.
\]
Finally we combine this result with the interpolation error estimate
\[
\|p_I^n -p^n\|_{1,\oGn,h}= h^\frac12 \|\nabla(p_I^n-p^n)\|_{\oGn}\lesssim h^{m+\frac12} \|p\|_{H^m(\oGn)} \lesssim h^{m+1} \|p\|_{H^m(\Gamma)}
\]
which then proves the result \eqref{Ep}.
\end{proof}

Both bounds in \eqref{FE_est1AR} and in \eqref{Ep} are optimal in terms of convergence order in $\Delta t$ and $h$.

%%%%%%%%%%%%%%%%%%%%%%%%%%%%%%%%%%%%%%%%%%%%%%%%%%%%%%%%%%%%%%%%%%%%%%%%%%%%%%%%%%%%%%%%%%%%%%%%%%%%%%%%%%%%%%%%%%%%%%
\section{Numerical experiments}\label{s:Numerics}
In this section, we present the results of numerical experiments with the proposed method. We consider a simple geometry evolution, namely a slowly  moving sphere which does not change its shape. The exact level set function is given by $\phi(t,\bx) = \norm{\bx - \bg(t)}^2 - 1$, with $\bg(t) = (0.2t, 0, 0)^T$. The normal velocity can be determined using $w_N =- \frac{\partial \phi}{\partial t}/\|\nabla \phi\|$. The exact tangential velocity $\bu_T$ is chosen as $\bu_T = \bP (\bn \times \nablaG \psi)$ with the stream function $\psi = xy - 2t$. Hence, $\divG \bu_T = 0$ holds.  The pressure is taken as $p = (x-0.2t) y + z$,  which  satisfies $\int_{\Gamma(t)} p \, ds =0$. The time interval is $I=[0,2]$ and the evolving sphere is embedded in the domain $\Omega =[- \tfrac43, \tfrac{10}{3}] \times [-\tfrac43, \tfrac43]^2$. Using Maple \cite{maple}, we calculate the corresponding exact force terms $\bf f$ and $f$.

The method introduced in this paper is implemented in Netgen/Ngsolve \cite{ngsolve} with the add-on ngsxfem \cite{ngsxfem}.
The geometry approximation is based on an Oswald type quasi-interpolation in the finite element space $V_{h,q}$ (cf. \eqref{eq:Vh}), denoted by $I_h^q$:
\begin{equation*}
    \phi_h^n = I_h^q (\phi(t_n,\cdot)), \quad \tilde\phi_h^n = I_h^{q+1} (\phi(t_n,\cdot)), \quad  \bn_h^n := \dfrac{\nabla \phi_h^n}{\norm{\nabla \phi_h^n}}, \quad  \tilde\bn_h^n := \dfrac{\nabla \tilde\phi_h^n}{\norm{\nabla \phi_h^n}}.
\end{equation*}
For the approximation of the Weingarten mapping we use $\bH_h^n = \nabla_{\Gamma_h} \overline{\bn}_h^n$, with $\overline{\bn}_h^n= I_h^{q} \bn_h^n$ the (Oswald type) componentwise quasi-interpolation of $\bn_h^n$ in the finite element space $(V_{h,q})^3$.
\\
In the experiments below we use the Taylor-Hood pair (cf. \eqref{TaylorHood}) with $m = 1$, $q=1$ or $q=2$, and $m=2$,  $q=3$. For the cases $q = 2$ and $q=3$ the zero level of  $\phi_h^n$ is not easy to determine, cf. Remark~\ref{remiso}, and we use the parametric finite element technique developed in \cite{grande2018analysis,lehrenfeld2016high}. For the time discretization, we use the BDF1, BDF2, or BDF3 scheme. We start with a uniform tetrahedral triangulation of $\Omega$ with maximum mesh size $h_0 = 0.5$. In each refinement step, the mesh size is halved. As the initial time step size, we chose $\Delta t_0 = 0.2$. In each spatial refinement step, we halve the time step size in the BDF2- and BDF3-scheme and we divide it by four when the BDF1 scheme is used.\\

The parameter choices we use are consistent with \eqref{CondA}--\eqref{CondB}. For the stabilization and the penalization we take $\rho_u = h^{-1}$, $\rho_p = h$ and $\tau = h^{-2}$. To define the extension area $\cO_\delta(\Gh)$ the approximative level set function $\phi_h$ (instead of the distance function) is used in \eqref{eq:def_narrow_band}. We choose the thickness of the narrow band as $\delta_n = \tilde c_{\delta_n} R \Delta t \norm{w_N}_{\infty, I_n}$, with $R=1,2,3,$ for BDF1, BDF2, and BDF3, respectively. The default value is $\tilde c_{\delta_n}=2.5$ and we check in each step whether $\delta_n$ is sufficiently large such that $\oGn \subset \cO_\delta(\Gamma_h^{n-i}),~ i=1,\ldots, R$, holds, cf. \eqref{cond1}.

In case of the BDF1 scheme we use the first order accurate linearization $(\nabla_{\Gamma_h} \bu^n)\bu_h^{n-1}$, cf. Section~\ref{s:discretization}. For the BDF2- and BDF3-scheme we consider the second and third order accurate linearization $(\nabla_{\Gamma_h} \bu_h^n)\bu_h^n \approx (\nabla_{\Gamma_h}\bu_h^n) \left(2\bu_h^{n-1}- \bu_h^{n-2}\right)$ and $(\nabla_{\Gamma_h} \bu_h^n)\bu_h^n \approx (\nabla_{\Gamma_h}\bu_h^n) \left(3\bu_h^{n-1} - 3\bu_h^{n-2} + \bu_h^{n-3} \right)$ respectively.

The error quantities we use are defined as follows. The velocity and pressure errors in each time step are denoted by $\err^n = \bu_h^n- \bu(t_n) $ and $r^n :=p_h^n- p(t_n)$ for $n = 1,\dots, N$. We consider the following error quantities:
\begin{eqnarray*}
	&\left(E_{\bu}^{L^2L^2}\right)^2 := \Delta t \sum_{n=1}^N \,\norm{\err^n}^2_{L^2(\Ghn)}, \qquad &\left(E_{\bu}^{L^2U}\right)^2 := \Delta t \sum_{n=1}^N \, \enormu{\err^n}^2, \\
    &\left(E_{p}^{L^2L^2}\right)^2 := \Delta t \sum_{n=1}^N \,\norm{r^n}^2_{L^2(\Ghn)}, &E_{p}^{L^1H^1_{\oGn}} := \Delta t \sum_{n=1}^N \, \|r^n\|_{1,{\oGn},h}.
	% \qquad & \left(E_{\bu}^{L^2H^1}\right)^2 := \Delta t \sum_{n=1}^N \,\norm{e_h^n}^2_{H^1(\Ghn)} ,\\
\end{eqnarray*}
Recall that the norm $\|\cdot\|_{1,{\oGn},h}$ is closely related (cf. discussion below \eqref{defpnorm}) to the \emph{scaled} $H^1(\Gamma_h^n)$ norm $h \|\nabla \cdot\|_{\Gamma_h^n}$, consistent with the error term $h^{m+1}$ in \eqref{Ep}.

We start with results for the BDF1 method, presented in the Figures~\ref{Fig_m1_q1_BDF1} and ~\ref{Fig_m1_q2_BDF1}. The first order convergence for the error $E_{\bu}^{L^2U}$ in Figure~\ref{Fig_m1_q1_BDF1}(a) shows that the geometric error bound  $h^q$ for this term in \eqref{FE_est1AR} is sharp. The result for the pressure error $E_{p}^{L^1H^1_{\oGn}}$ in  Figure~\ref{Fig_m1_q1_BDF1}(b) indicates that the geometric error term $h^q$ in \eqref{Ep} might be not sharp. The results for $E_{\bu}^{L^2U}$ and $E_{p}^{L^1H^1_{\oGn}}$ in Figure~\ref{Fig_m1_q2_BDF1} confirm the second order convergence predicted by Theorem~\ref{Th2AR}. Although we did not derive a bound for the error $E_{\bu}^{L^2L^2}$, this error has the expected (approximately) second order convergence in Figure~\ref{Fig_m1_q2_BDF1}(a), due to the term $\Delta t \sim h^2$ in the time discretization error.

\newcommand{\errorfile}{NumericalResults/m1_q1_BDF1.dat}

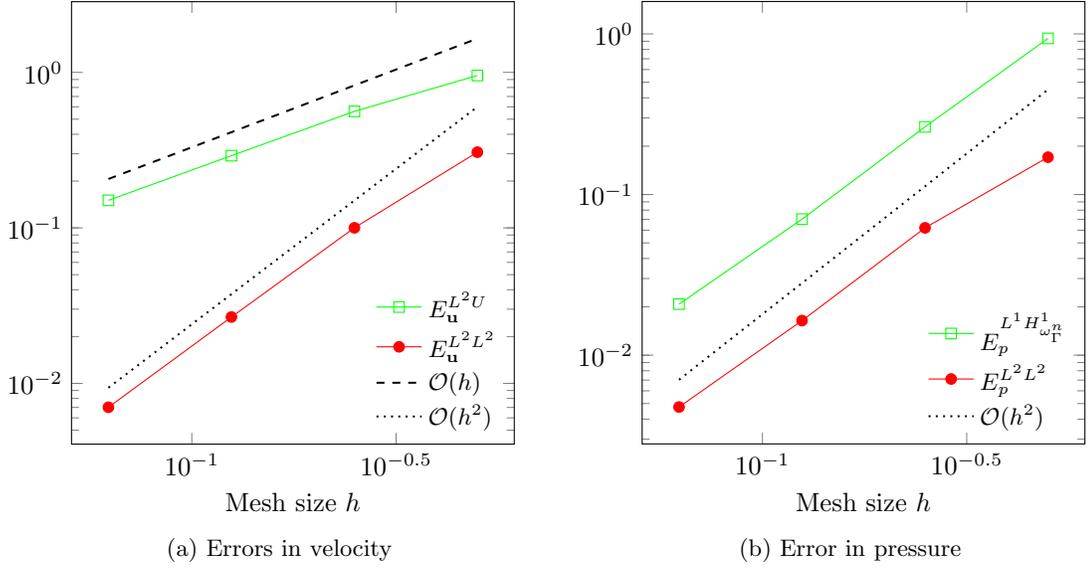
\begin{figure}[H]
	\centering
	\begin{subfigure}[b]{0.49\textwidth}
		\begin{tikzpicture}
			\begin{axis}
				[width=\linewidth,height=\linewidth,xlabel={Mesh size $h$}, domain={0.0625:0.5}, legend style={at={(.99,0.01)}, anchor=south east, legend columns=1, draw=none, fill=none}, legend cell align=left, cycle list name=mark list, ymode=log, xmode=log]
				\addplot[green, mark=square] table[x=MeshSize, y=L2Energyu]{\errorfile}; \addlegendentry{\small $E_{\bu}^{L^2U}$}
				\addplot[red, mark=*] table[x=MeshSize, y=L2L2u]{\errorfile}; \addlegendentry{\small $E_{\bu}^{L^2L^2}$};
				\addplot[dashed,line width=0.75pt]{3.3*x}; \addlegendentry{\small $\cO(h)$}
				\addplot[dotted,line width=0.75pt]{2.4*x^2}; \addlegendentry{\small $\cO(h^2)$}
			\end{axis}
		\end{tikzpicture}
		\caption{Errors in velocity}
	\end{subfigure}
	\begin{subfigure}[b]{0.49\textwidth}
		\begin{tikzpicture}
			\begin{axis}
				[width=\linewidth,height=\linewidth,xlabel={Mesh size $h$}, domain={0.0625:0.5}, legend style={at={(.99,0.01)}, 	anchor=south east, legend columns=1, draw=none, fill=none}, legend cell align=left, cycle list name=mark list, ymode=log, xmode=log]
				\addplot[green, mark=square] table[x=MeshSize, y=L1Energyp]{\errorfile}; \addlegendentry{\small $E_{p}^{L^1H^1_{\oGn}}$}
				\addplot[red, mark=*] table[x=MeshSize, y=L2L2p]{\errorfile}; \addlegendentry{\small $E_{p}^{L^2L^2}$};
				\addplot[dotted,line width=0.75pt]{1.8*x^2}; \addlegendentry{\small $\cO(h^2)$}
			\end{axis}
		\end{tikzpicture}
		\caption{Error in pressure}
	\end{subfigure}
	\caption{$m=1$, $q=1$, BDF $=1$, $\Delta t \sim h^2$}
	 \label{Fig_m1_q1_BDF1}
\end{figure}

\renewcommand{\errorfile}{NumericalResults/m1_q2_BDF1.dat}

\begin{figure}[H]
	\centering
	\begin{subfigure}[b]{0.49\textwidth}
		\begin{tikzpicture}
			\begin{axis}
				[width=\linewidth,height=\linewidth,xlabel={Mesh size $h$}, domain={0.0625:0.5}, legend style={at={(.99,0.01)}, anchor=south east, legend columns=1, draw=none, fill=none}, legend cell align=left, cycle list name=mark list, ymode=log, xmode=log]
				\addplot[green, mark=square] table[x=MeshSize, y=L2Energyu]{\errorfile};\addlegendentry{\small $E_{\bu}^{L^2U}$}
				\addplot[red, mark=*] table[x=MeshSize, y=L2L2u]{\errorfile};\addlegendentry{\small $E_{\bu}^{L^2L^2}$};
				\addplot[dotted,line width=0.75pt]{1.1*x^2}; \addlegendentry{\small $\cO(h^2)$}
			\end{axis}
		\end{tikzpicture}
		\caption{Errors in velocity}
	\end{subfigure}
	\begin{subfigure}[b]{0.49\textwidth}
		\begin{tikzpicture}
			\begin{axis}
				[width=\linewidth,height=\linewidth,xlabel={Mesh size $h$}, domain={0.0625:0.5}, legend style={at={(.99,0.01)}, 	anchor=south east, legend columns=1, draw=none, fill=none}, legend cell align=left, cycle list name=mark list, ymode=log, xmode=log]
				\addplot[green, mark=square] table[x=MeshSize, y=L1Energyp]{\errorfile}; \addlegendentry{\small $E_{p}^{L^1H^1_{\oGn}}$}
				\addplot[red, mark=*] table[x=MeshSize, y=L2L2p]{\errorfile}; \addlegendentry{\small $E_{p}^{L^2L^2}$};
				\addplot[dotted,line width=0.75pt]{1.2*x^2}; \addlegendentry{\small $\cO(h^2)$}
			\end{axis}
		\end{tikzpicture}
		\caption{Error in pressure}
	\end{subfigure}
	\caption{$m=1$, $q=2$, BDF $=1$, $\Delta t \sim h^2$}
	\label{Fig_m1_q2_BDF1}
\end{figure}
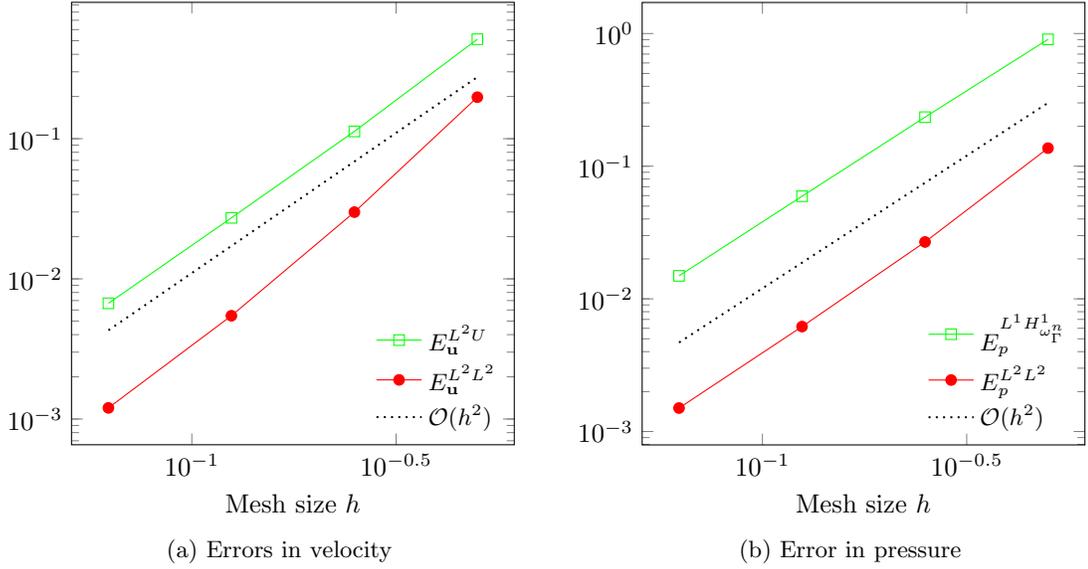

Results for the BDF2 method are shown in Figures~\ref{Fig_m1_q1_BDF2}, and \ref{Fig_m1_q2_BDF2}. We claim that our theoretical analysis can be extended to this case, in which then the time discretization error terms $\Delta t$ in \eqref{FE_est1AR}-\eqref{Ep} would be replaced by $\Delta t^2$. In Figure~\ref{Fig_m1_q1_BDF2}(a) we observe again that for $E_{\bu}^{L^2U}$  we only have first order convergence due to the $h^1$ geometric error term. The results in Figure~\ref{Fig_m1_q2_BDF2} show close to second order convergence for all error quantities (as in Figure~\ref{Fig_m1_q2_BDF1}).

\renewcommand{\errorfile}{NumericalResults/m1_q1_BDF2.dat}
\begin{figure}[H]
	\centering
	\begin{subfigure}[b]{0.49\textwidth}
		\begin{tikzpicture}
			\begin{axis}
				[width=\linewidth,height=\linewidth,xlabel={Mesh size $h$}, domain={0.03125:0.5}, legend style={at={(.99,0.01)}, anchor=south east, legend columns=1, draw=none, fill=none}, legend cell align=left, cycle list name=mark list, ymode=log, xmode=log]
				\addplot[green, mark=square] table[x=MeshSize, y=L2Energyu]{\errorfile};\addlegendentry{\small $E_{\bu}^{L^2U}$}
				\addplot[red, mark=*] table[x=MeshSize, y=L2L2u]{\errorfile};\addlegendentry{\small $E_{\bu}^{L^2L^2}$};
				\addplot[dashed,line width=0.75pt]{5.5*x}; \addlegendentry{\small $\cO(h)$}
				\addplot[dotted,line width=0.75pt]{0.8*x^2}; \addlegendentry{\small $\cO(h^2)$}
			\end{axis}
		\end{tikzpicture}
		\caption{Errors in velocity}
	\end{subfigure}
	\begin{subfigure}[b]{0.49\textwidth}
		\begin{tikzpicture}
			\begin{axis}
				[width=\linewidth,height=\linewidth,xlabel={Mesh size $h$}, domain={0.03125:0.5}, legend style={at={(.99,0.01)}, 	anchor=south east, legend columns=1, draw=none, fill=none}, legend cell align=left, cycle list name=mark list, ymode=log, xmode=log]
				\addplot[green, mark=square] table[x=MeshSize, y=L1Energyp]{\errorfile}; \addlegendentry{\small $E_{p}^{L^1H^1_{\oGn}}$}
				\addplot[red, mark=*] table[x=MeshSize, y=L2L2p]{\errorfile}; \addlegendentry{\small $E_{p}^{L^2L^2}$};
				\addplot[dashed,line width=0.75pt]{3.1*x}; \addlegendentry{\small $\cO(h)$}
				\addplot[dotted,line width=0.75pt]{2.5*x^2}; \addlegendentry{\small $\cO(h^2)$}
			\end{axis}
		\end{tikzpicture}
		\caption{Error in pressure}
	\end{subfigure}
	\caption{$m=1$, $q=1$, BDF $=2$, $\Delta t \sim h$}
	\label{Fig_m1_q1_BDF2}
\end{figure}
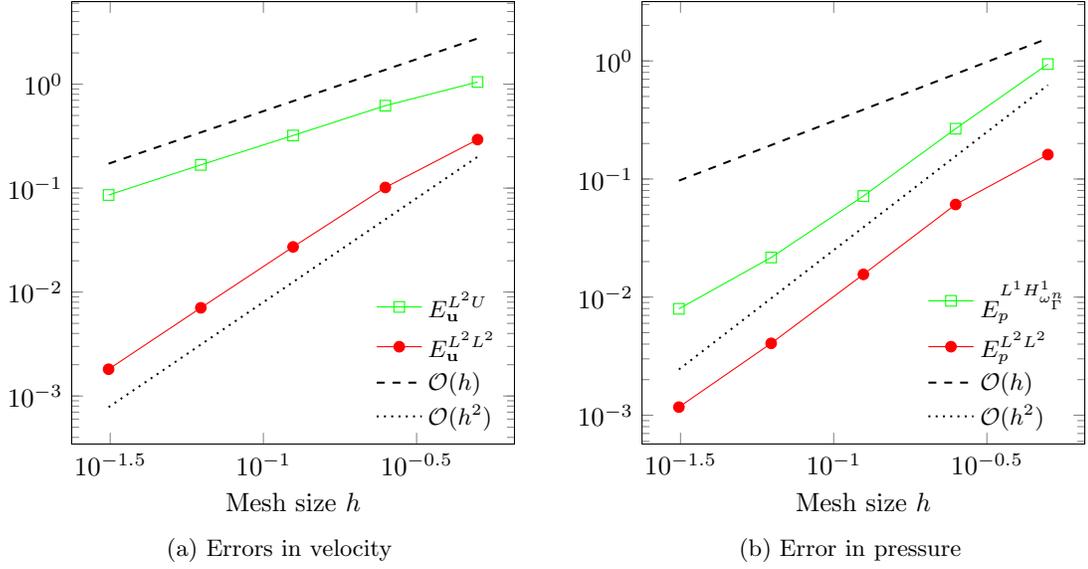

\renewcommand{\errorfile}{NumericalResults/m1_q2_BDF2.dat}
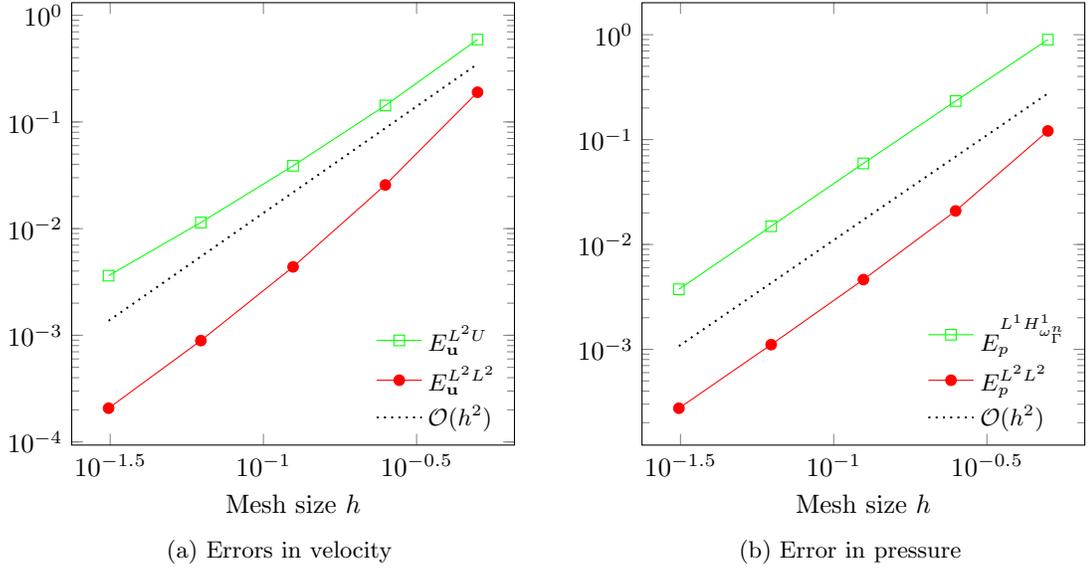
\begin{figure}[H]
	\centering
	\begin{subfigure}[b]{0.49\textwidth}
		\begin{tikzpicture}
			\begin{axis}
				[width=\linewidth,height=\linewidth,xlabel={Mesh size $h$}, domain={0.03125:0.5}, legend style={at={(.99,0.01)}, anchor=south east, legend columns=1, draw=none, fill=none}, legend cell align=left, cycle list name=mark list, ymode=log, xmode=log]
				\addplot[green, mark=square] table[x=MeshSize, y=L2Energyu]{\errorfile};\addlegendentry{\small $E_{\bu}^{L^2U}$}
				\addplot[red, mark=*] table[x=MeshSize, y=L2L2u]{\errorfile};\addlegendentry{\small $E_{\bu}^{L^2L^2}$};
				\addplot[dotted,line width=0.75pt]{1.4*x^2}; \addlegendentry{\small $\cO(h^2)$}
			\end{axis}
		\end{tikzpicture}
		\caption{Errors in velocity}
	\end{subfigure}
	\begin{subfigure}[b]{0.49\textwidth}
		\begin{tikzpicture}
			\begin{axis}
				[width=\linewidth,height=\linewidth,xlabel={Mesh size $h$}, domain={0.03125:0.5}, legend style={at={(.99,0.01)}, 	anchor=south east, legend columns=1, draw=none, fill=none}, legend cell align=left, cycle list name=mark list, ymode=log, xmode=log]
				\addplot[green, mark=square] table[x=MeshSize, y=L1Energyp]{\errorfile}; \addlegendentry{\small $E_{p}^{L^1H^1_{\oGn}}$}
				\addplot[red, mark=*] table[x=MeshSize, y=L2L2p]{\errorfile}; \addlegendentry{\small $E_{p}^{L^2L^2}$};
				\addplot[dotted,line width=0.75pt]{1.1*x^2}; \addlegendentry{\small $\cO(h^2)$}
			\end{axis}
		\end{tikzpicture}
		\caption{Error in pressure}
	\end{subfigure}
	\caption{$m=1$, $q=2$, BDF $=2$, $\Delta t \sim h$}
	 \label{Fig_m1_q2_BDF2}
\end{figure}

We also consider the higher-order Taylor-Hood pair with $m=2$. We combine it with a higher  geometric order ($q=3$) and the higher order time integration scheme BDF3. All parameters are taken the same as in the  experiments above. The results are shown in figure~\ref{Fig_m2_q3_BDF3}. We observe  optimal third order convergence in the energy norm for the velocity and in the scaled $H^1(\Gamma_h^n)$-norm for the pressure. 

\renewcommand{\errorfile}{NumericalResults/m2_q3_BDF3.dat}
\begin{figure}[ht!]
	\centering
	\begin{subfigure}[b]{0.49\textwidth}
		\begin{tikzpicture}
			\begin{axis}
				[width=\linewidth,height=\linewidth,xlabel={Mesh size $h$}, domain={0.03125:0.5}, legend style={at={(.99,0.01)}, anchor=south east, legend columns=1, draw=none, fill=none}, legend cell align=left, cycle list name=mark list, ymode=log, xmode=log]
				\addplot[green, mark=square] table[x=MeshSize, y=L2Energyu]{\errorfile};\addlegendentry{\small $E_{\bu}^{L^2U}$}
				\addplot[red, mark=*] table[x=MeshSize, y=L2L2u]{\errorfile};\addlegendentry{\small $E_{\bu}^{L^2L^2}$};
				\addplot[dotted,line width=0.75pt]{0.65*x^3}; \addlegendentry{\small $\cO(h^3)$}
			\end{axis}
		\end{tikzpicture}
		\caption{Errors in velocity}
	\end{subfigure}
	\begin{subfigure}[b]{0.49\textwidth}
		\begin{tikzpicture}
			\begin{axis}
				[width=\linewidth,height=\linewidth,xlabel={Mesh size $h$}, domain={0.03125:0.5}, legend style={at={(.99,0.01)}, 	anchor=south east, legend columns=1, draw=none, fill=none}, legend cell align=left, cycle list name=mark list, ymode=log, xmode=log]
				\addplot[green, mark=square] table[x=MeshSize, y=L1Energyp]{\errorfile}; \addlegendentry{\small $E_{p}^{L^1H^1_{\oGn}}$}
				\addplot[red, mark=*] table[x=MeshSize, y=L2L2p]{\errorfile}; \addlegendentry{\small $E_{p}^{L^2L^2}$};
				\addplot[dotted,line width=0.75pt]{0.9*x^3}; \addlegendentry{\small $\cO(h^3)$}
			\end{axis}
		\end{tikzpicture}
		\caption{Error in pressure}
	\end{subfigure}
	\caption{$m=2$, $q=3$, BDF $=3$, $\Delta t \sim h$}
	\label{Fig_m2_q3_BDF3}
\end{figure}
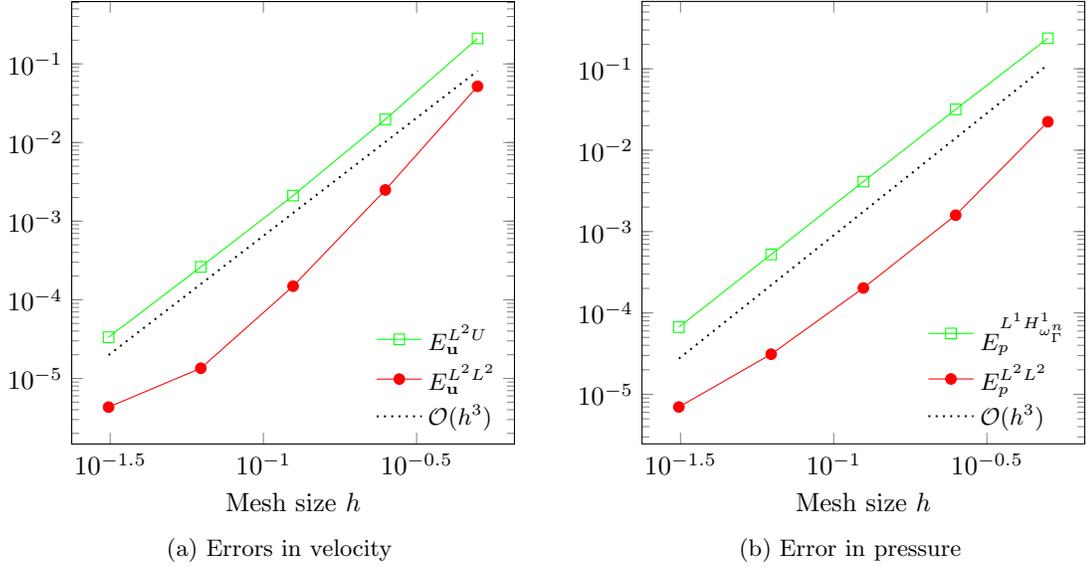

\appendix
\section{Proof of \eqref{eqv1}, \eqref{intestimate} and \eqref{Kornestimate}} \label{AppSecA}
%Fix any $n\in\{1,\dots,N\}$,  and
Define $\bB=\bB(t,\bx)=\bP(\bI-d\bH)\bP_h$ for $\bx\in\Gamma_h(t)$. Geometry approximation assumptions from section~\ref{s:geom} imply
that $\bB$ is invertible on the range of $\bP$ (for $h$ small enough, cf. \eqref{CondA}) and the following uniform in time and discretization parameters estimates hold (see \cite{hansbo2020analysis}):
\begin{equation}\label{eq:aux1255}
\begin{split}
\|\bB\|_{L^\infty(\Gamma_h)} & \le C,\quad \|\bP_h\bB^{-1}\bP\|_{L^\infty(\Gamma_h)}\le C, \\ \|\bP_h\bB^{-1}\bP-\bP_h\bP\|_{L^\infty(\Gamma_h)} & \le Ch^{q+1}, \quad \|1-|\mbox{det}(\bB)|\,\|_{L^\infty(\Gamma_h)}\le Ch^{q+1}.
\end{split}
\end{equation}
These bounds and the identities
\begin{equation}\label{eq:aux1243}
d\Gamma=|\mbox{det}(\bB)|d\Gamma_h\quad\text{and}\quad \nabla_{\Gamma}u^\ell(\bp(\bx))=\bP\bB^{-T}\nabla_{\Gamma_h}u(\bx),~\bx\in\Gamma_h,~u\in H^1(\Gamma_h),
\end{equation}
and the vector analogues, see \cite[section~5.4.1]{Jankuhn2020}, imply the norm equivalences in \eqref{eqv1}. Same uniform estimates \eqref{eq:aux1255} together with \eqref{eq:normals} and \eqref{discrWeingarten} yield (cf.~\cite[Lemma~5.14]{Jankuhn2020}) for any $\bv\in H^1(\Gamma_h)^3$,
\begin{equation}\label{eq:aux1263}
\begin{split}
\|(\nabla_\Gamma \bv^\ell)^e-\nabla_{\Gamma_h} \bv \|_{\G{n}{h}} &\lesssim h^q \|\bv\|_{H^1(\G{n}{h})},\\
	 \|E_s(\bP\bv^\ell)^e-E_h(\bP_h\bv)\|_{L^2(\Gamma_h)}&\lesssim h^q(\|\bv\|_{H^1(\Gamma_h)}+h^{-1}\|\bv\cdot\bn_h\|_{L^2(\Gamma_h)}.
\end{split}
\end{equation}
Now the uniform Korn-type inequality follows from the uniform in time Korn inequality on $\Gamma(t)$ from \cite[Lemma~3.2]{olshanskii2022tangential}, estimates \eqref{eqv1}, \eqref{eq:aux1263} and FE trace and inverse inequalities by the arguments in \cite[Lemma~5.16]{Jankuhn2020}.
Finally, \eqref{intestimate} follows from the uniform in time interpolation inequality on $\Gamma(t)$ from \cite[Lemma~3.4]{olshanskii2022tangential} and \eqref{eqv1}.

\section{Proof of Lemma~\ref{lem2b}} \label{Appsec2}
To prove Lemma~\ref{lem2b} we first note that due to condition \eqref{cond1} we have $\G{n}{h}\subset U_{\delta}(\G{n-1}{h})\subset\mathcal{O}_{\delta}(\G{n-1}{h})\subset \mathcal{O}(\G{n-1}{})$, where the last inclusion holds for $\Delta t$, $h$ small enough, see condition \eqref{CondA}. Hence, for $v\in L^2(\G{n-1}{h})$ we define	a lift $v^\ell\in L^2\big(U_{\delta}(\G{n-1}{h}) \big)$ as in \eqref{deflifting}, with $n$ replaced by $n-1$.
%follows. For any $\by\in U_{\delta}(\G{n-1}{h})$ we define $v^\ell(\by):=v(\bx)$ for $\bx\in\G{n-1}{h}$ such that $\bp^{n-1}(\by)=\bp^{n-1}(\bx)$. Recall that $\bp^{n-1}$ is the closest point projection on $\Gamma^{n-1}$.
We use the splitting
	\begin{equation}\label{eq:aux514}
		\|v_h\|_{\G{n}{h}}^2 = \int_{\G{n}{h}}(|v_h|^2-|v_h^\ell|^2)\,ds_h+\|v_h^\ell\|_{\G{n}{h}}^2.
	\end{equation}
For the second term on the right-hand side we apply  the estimate $\|v_h^\ell\|_{\G{n}{h}} \lesssim \  \|v_h\|_{\G{n-1}{h}}$ (see, e.g.  \cite[Lemma 6]{lehrenfeld2018stabilized}). 	For the first term we obtain (we abbreviate $U_{\delta} = U_{\delta}(\G{n-1}{h})$ here):
	\begin{align*}
		\int_{\G{n}{h}}&(|v_h|^2-|v_h^\ell|^2)\,ds_h
		\lesssim  \int_{U_{\delta}} \left| \bn^{n-1}\cdot\nabla(|v_h|^2-|v_h^\ell|^2) \right| \,dx
		&
		({\footnotesize |v_h|^2=|v_h^\ell|^2~\text{on}~\G{n-1}{h}} )
		\\
		\leq
		&
		\, \int_{U_{\delta}} \left| \bn_h^{n-1}\cdot\nabla|v_h|^2 \right| \,dx
		+\, \int_{U_{\delta}} \left| (\bn^{n-1}-\bn_h^{n-1})\cdot\nabla|v_h|^2 \right| \,dx
		&
		({\footnotesize \text{as } \bn^{n-1}\cdot\nabla |v_h^\ell|^2\!\! = 0} )
		\\
		\lesssim
		&
		~\|v_h\|_{U_{\delta}}\|\bn_h^{n-1}\cdot\nabla v_h\|_{{U_{\delta}}}
		+\|\bn^{n-1}-\bn_h^{n-1}\|_{L^\infty (U_{\delta})} \|\nabla v_h\|_{U_{\delta}}\|v_h\|_{U_{\delta}}
		\\
		\lesssim
		&
		~\|v_h\|_{U_{\delta}} \big(\|\bn_h^{n-1}\cdot\nabla v_h\|_{{U_{\delta}}} + h^{q-1} \|v_h\|_{\mathcal{O}_{\delta}(\G{n-1}{h})} \big) & ({\footnotesize  \text{eq. \eqref{eq:normals} and FE inv. ineq.}})
		\\
		\lesssim
		&  ~\|v_h\|_{U_{\delta}} \big(\|\bn_h^{n-1}\cdot\nabla v_h\|_{{U_{\delta}}} +\| v_h\|_{\G{n-1}{h}}+ \|\bn_h^{n-1}\cdot\nabla v_h\|_{\mathcal{O}_\delta (\G{n-1}{h})} \big) &
		({\footnotesize \text{eq. \eqref{fund1} and $\delta_{n-1}+h \lesssim 1$}})
		\\
		\lesssim
		& ~ \big( \delta_{n-1}^\frac12 \| v_h\|_{\G{n-1}{h}} + \delta_{n-1}\|\bn_h^{n-1}\cdot\nabla v_h\|_{\mathcal{O}_\delta (\G{n-1}{h})}\big)&({\footnotesize \text{eq. \eqref{fund1a}}}) \\
		&\qquad\times \big(\| v_h\|_{\G{n-1}{h}}+ \|\bn_h^{n-1}\cdot\nabla v_h\|_{\mathcal{O}_\delta (\G{n-1}{h})} \big) &
		\\
		\lesssim
		&~\|v_h\|_{\G{n-1}{h}}^2 + \delta_{n-1} \|\bn_h^{n-1} \cdot \nabla v_h\|_{\mathcal{O}_\delta(\G{n-1}{h})}^2 .&
	\end{align*}
	This completes the proof.
\section{Proof of \eqref{RESI5}} \label{AppendixC}
The term that has to be bounded is given by
\[ A:=\left|\int_{\G{n}{}}  w_N^n\bv_h^\ell \cdot \bH\bu^n\, ds - \int_{\G{n}{h}} w_N^{e,n} \bv_h \cdot \bH_h \bu^n\, ds_h \right|.
\]
Using $\bH = \bP\nabla \bn \bP$, $\bP \bu^n=\bu^n$ and the product rule $\nabla (\bw \cdot \bn)^T= \bw \cdot \nabla \bn + \bn \cdot \nabla \bw$ we get
\begin{equation} \label{eq45} \begin{split}
& \int_{\G{n}{}} w_N^n\bv_h^\ell \cdot \bH\bu^n\, ds  = \int_{\G{n}{}} w_N^n \nabla(\bP \bv_h^\ell \cdot \bn) \cdot \bu^n  \, ds - \int_{\G{n}{}} w_N^n \bn \cdot \nabla (\bP \bv_h^\ell) \bu^n \, ds \\
  & = \int_{\G{n}{}} w_N^n \nabla_\Gamma (\bv_h^\ell \cdot \bP \bn) \cdot \bu^n  \, ds - \int_{\G{n}{}} w_N^n \bn \cdot \nabla (\bP \bv_h^\ell) \bu^n \, ds =
    - \int_{\G{n}{}} w_N^n \bn \cdot \nabla (\bP \bv_h^\ell) \bu^n \, ds.
 \end{split}
\end{equation}
%We consider the $\int_{\G{n}{h}} \cdot \, ds_h$ term. For $\by$ from a sufficiently small strip around $\Gamma^n$ we have $\nabla \normalbar^\ell (\by) = \nabla \normalbar^\ell (\bp(\by)) (\bP(\by) -d(\by) \bH(\by))$, with $\bp$ the closest point projection on $\Gamma^n$ and $d$ the signed distance function to $\Gamma^n$.
Using  the definition $\bH_h = \nabla_{\Gamma_h} \normalbar = \bP_h \nabla \normalbar^\ell \bP_h$,  the transformation formula \eqref{transform} applied to $\normalbar$,  $\|d\|_{L^\infty(\Gamma_h^n)} \lesssim h^{q+1}$, $\|\bP-\bP_h\|_{L^\infty(\Gamma_h^n)} \lesssim h^q$, the  bound for the surface measure change in \eqref{eq:aux1255} and the smoothness of $\bu$ we obtain
\begin{equation} \label{eq46}
 \left|\int_{\G{n}{h}} w_N^{e,n} \bv_h \cdot \bH_h \bu^n\, ds_h - \int_{\Gamma^n} w_N^{n} \bv_h^\ell \cdot \bP \nabla \normalbar^\ell \bP \bu^n \, ds\right| \lesssim h^q \|\bv_h\|_{\Gamma_h^n}.
\end{equation}
We apply partial integration as in \eqref{eq45}, which yields
\begin{equation} \label{eq47}
 \int_{\Gamma^n} w_N^{n} \bv_h^\ell \cdot \bP \nabla \normalbar^\ell \bP \bu^n \, ds =  \int_{\G{n}{}} w_N^n \nabla_\Gamma (\bv_h^\ell \cdot \bP \normalbar^\ell) \cdot \bu^n  \, ds - \int_{\G{n}{}} w_N^n \normalbar^\ell \cdot \nabla (\bP \bv_h^\ell) \bu^n \, ds.
\end{equation}
For the first term on the right-hand side we apply partial integration and using $\|\bP \normalbar^\ell\|_{L^\infty(\Gamma^n)} = \|\bP (\normalbar^\ell - \bn)\|_{L^\infty(\Gamma^n)} \lesssim h^q$, cf. \eqref{discrWeingarten}, we get
\begin{equation} \label{eq48}
\left| \int_{\G{n}{}} w_N^n \nabla_\Gamma (\bv_h^\ell \cdot \bP \normalbar^\ell) \cdot \bu^n  \, ds\right| = \left | \int_{\G{n}{}} (\bv_h^\ell \cdot \bP \normalbar^\ell) \DivG (w_N^n \bu^n) \, ds\right| \lesssim h^q \|\bv_h\|_{\Gamma_h^n}.
\end{equation}
With the  results \eqref{eq45}-\eqref{eq48} we get, using again \eqref{discrWeingarten},
\begin{align*}
  A & \leq \left|\int_{\G{n}{}} w_N^n \bn \cdot \nabla (\bP \bv_h^\ell) \bu^n \, ds - \int_{\G{n}{}} w_N^n \normalbar^\ell \cdot \nabla (\bP \bv_h^\ell) \bu^n \, ds  \right| +h^q \|\bv_h\|_{\Gamma_h^n} \\
   & = \left|\int_{\G{n}{}} w_N^n (\bn- \normalbar^\ell) \cdot \nabla (\bP \bv_h^\ell) \bu^n \, ds\right| +h^q \|\bv_h\|_{\Gamma_h^n} \lesssim
    h^q \|\bv_h\|_{H^1(\Gamma_h^n)},
\end{align*}
which completes the proof.

\subsection*{Acknowledgment} The authors A. Reusken and P. Schwering wish to thank the German Research Foundation (DFG) for financial support within the Research Unit ``Vector- and tensor valued surface PDEs'' (FOR 3013) with project no. RE 1461/11-2. The author M.Olshanskii was partially supported by US National Science Foundation (NSF) through DMS-2011444.
% \vspace*{-0.25cm}
\bibliographystyle{siam}
\bibliography{literatur}{}
\end{document}